\documentclass{amsart}
\usepackage{color}
\usepackage{amsfonts, amsmath, amssymb, amscd, amsthm, bm}
\usepackage{multirow} 

\newtheorem{them}{Theorem}[section]
\newtheorem{lem}{Lemma}[section]

\newtheorem*{thm}{Main Theorem}

\newtheorem{definition}{Definition}

\theoremstyle{remark}
\newtheorem{rem}{Remark}[section]

\numberwithin{equation}{section}

\begin{document}
\title[On product minimal Lagrangian submanifolds]
{On product minimal Lagrangian submanifolds in complex space forms}

\author[X. Cheng, Z. Hu, M. Moruz and L. Vrancken]
{Xiuxiu Cheng, Zejun Hu, Marilena Moruz and Luc Vrancken}


\thanks{The first two authors were supported by NSF of China,
Grant Number 11771404. The third author is a postdoctoral fellow
of FWO - Flanders, Belgium.}



\begin{abstract}
In this paper we consider minimal Lagrangian submanifolds in
$n$-dimensional complex space forms. More precisely, we study
such submanifolds which, endowed with the induced metrics,
write as a Riemannian product of two Riemannian
manifolds, each having constant sectional curvature. As the main
result, we give a complete classification of these submanifolds.
\end{abstract}

\maketitle

\vskip 20pt
\section{Introduction}\label{sect:1}

In this paper we study Lagrangian submanifolds of the $n$-dimensional 
complex space forms $\tilde{M}^n(4\tilde c)$ of constant holomorphic 
sectional curvature $4\tilde c$. Indeed, the complex space forms are 
the easiest examples of K\"ahler manifolds. These are almost Hermitian 
manifolds for which the almost complex structure $J$ is parallel with 
respect to the Levi-Civita connection $\nabla$ of the Hermitian metric 
$g$. The standard models of complex space forms are the complex projective 
space $\mathbb{CP}^n$, the complex Euclidean space $\mathbb{C}^n$ and the
complex hyperbolic space $\mathbb{CH}^n$, according to whether the
holomorphic sectional curvature satisfies $\tilde{c}>0$, $\tilde{c}=0$ 
or $\tilde{c}<0$.

There are two special classes of submanifolds of a K\"ahler manifold
depending on the behavior of the complex structure $J$ with respect
to the submanifold.

A submanifold $M$ of $\tilde{M}^n$ is called almost complex 
if and only if $J$ maps tangent vectors to tangent vectors. On the other 
hand, $M$ is called totally real if the almost complex structure $J$ of
$\tilde{M}^n$ carries each tangent space of $M$ into its 
corresponding normal space. The study of minimal totally real submanifolds
originates with the work of Chen and Ogiue (see \cite{CO}). A
special case here happens when the real dimension of the submanifold
equals the complex dimension of the ambient space. In that case $J$
interchanges the tangent and the normal spaces. Such submanifolds
are called Lagrangian submanifolds. These can also be seen as
submanifolds of the largest possible dimension on which the
symplectic form vanishes identically.

For the study of minimal Lagrangian immersions in complex space
forms one may find a short survey in \cite{C}, where some of the
main results are mentioned (see also for example \cite{BR,CLU,CU,
CDVV,CO,CV,DT,DLVW,E,KV,LW1,LW3,MO,WLV1,WLV2}).

The fundamental question in submanifold theory is then to determine
to what extent the geometry of the submanifold determines the
immersion of the submanifold in the ambient space. In that respect,
it was shown by Ejiri \cite{E} that an $n$-dimensional Lagrangian
minimal submanifold of constant sectional curvature $c$ immersed in
an $n$-dimensional complex space form is either totally geodesic or
flat $(c=0)$ (cf. also \cite{LZ} and \cite{DT}). More precisely in
the latter case it must be congruent to a specific Lagrangian tori
in the complex projective space (see Main Theorem below). Note that
the condition that the immersion is minimal is unavoidable. From
\cite{CDVV} and \cite{T} we can see that one cannot expect to
obtain a general classification of all Lagrangian submanifolds of
real space forms in complex space forms.

In this paper we consider the logical next step. We will assume that
our manifold $M$ is isometric with $M_1^{n_1}(c_1)\times
M_2^{n_2}(c_2)$, i.e. it is a product of two real space forms of
constant sectional curvature, respectively $c_1$ and $c_2$. As the
main result of the paper we extend Ejiri's result by proving
\begin{thm}\label{maintheorem}
Let $\psi :M^n\to \tilde M^n(4\tilde{c})$ be a minimal Lagrangian
immersion into a complex space form with induced metric $\langle\cdot,\cdot\rangle$.
If $M^n=M_1^{n_1}(c_1)\times M_2^{n_2}(c_2)$, where $n=n_1+n_2$,
$M_1^{n_1}(c_1)$ (resp. $M_1^{n_2}(c_2)$) is an $n_1$
(resp. $n_2$)-dimensional Riemannian manifold of constant sectional
curvature $c_1$ (resp. $c_2$), then $c_1 c_2=0$. Moreover,
\begin{enumerate}
\item if $c_1=c_2=0$, then $M^n$ is equivalent to either
the totally geodesic immersion in $\mathbb{C}^n$ or the Lagrangian
flat torus in $\mathbb{CP}^n(4\tilde{c})$.
\item if $c_1 c_2=0$ and $c_1^2+c_2^2\neq 0$, without loss of generality, we may
assume that $c_1=0$ and $c_2\neq0$. Then we have
$c_2=\frac{n_1+n_2+1}{n_2+1}\tilde{c}>0$, say $\tilde{c}=1$, so the
ambient space is $\mathbb{CP}^n(4)$, and the immersion is congruent
with
\begin{equation*}
\tfrac{1}{n+1}(e^{iu_1},\ldots,e^{iu_{n_1}}, a e^{i u_{n_1+1}}y_1,
\ldots,a e^{i u_{n_1+1}}y_{n_2+1} ),
\end{equation*}
\end{enumerate}
where
\begin{enumerate}
\item [(i)] $(y_1,y_2,\ldots,y_{n_2+1})$ describes the standard sphere
$\mathbb{S}^{n_2}\hookrightarrow \mathbb{R}^{n_2+1}\hookrightarrow
\mathbb{C}^{n_2+1}$,

\item [(ii)] $a=\sqrt{n_2+1}$,

\item [(iii)] $u_1+\cdots+u_{n_1}+a^2 u_{n_1+1}=0$.
\end{enumerate}
\end{thm}

\begin{rem}\label{rm:1.1}
The technique we use in the proof of the Main Theorem is based on two
steps. The first step is to take cyclic permutation of the covariant
derivative of the Codazzi equation. The second step is then expressing
the second fundamental form of the submanifold $M^n$ with respect to a
conveniently chosen frame. To do so, we proceed by induction (see
\cite{VLS} and \cite{CHMV}). One should notice that, eventually, our
main result follows directly from the theorems in \cite{LW2}.
\end{rem}

\section{Preliminaries}\label{sect:2}

In this section, we will recall the basic formulas for Lagrangian
submanifolds in complex space forms. Let $\tilde{M}^n(4\tilde{c})$
be a complex space form of complex dimension $n$ and having constant
holomorphic sectional curvature $4\tilde{c}$. Let $M^n$ be a minimal
Lagrangian submanifold in $\tilde{M}^n(4\tilde{c})$ given by the
immersion $\psi: M^n\to \tilde{M}^n(4\tilde{c})$ such that
\begin{equation}\label{eqn:2.1}
M^n=M_1^{n_1}(c_1)\times M_2^{n_2}(c_2),
\end{equation}
where $n_1+n_2=n$, $M_1^{n_1}(c_1)$ and $M_2^{n_2}(c_2)$ are
manifolds of real dimensions $n_1$ and $n_2$ and have constant
sectional curvature $c_1$ and $c_2$, respectively.

Let $\nabla$ and $\tilde\nabla$ be the Levi-Civita connections
on $M^n$ and $\tilde{M}^n(4\tilde{c})$, respectively. The
formulas of Gauss and Weingarten write out as
\begin{equation}\label{eqn:2.2}
\tilde\nabla_XY=\nabla_XY+h(X,Y), \quad \tilde\nabla_X\xi=-A_{\xi}X+\nabla^{\perp}_X\xi,
\end{equation}
for $X,Y$ tangent to $M^n$ and $\xi$ normal to $M^n$, where
$h,A$ and $\nabla^\perp$ are the second fundamental form,
the shape operator and the normal connection, respectively.

Notice that we will always identify $M^n$ with its immersed
image in $\tilde{M}^n(4\tilde{c})$. As $M^n$ is Lagrangian,
we have that the almost complex structure $J$ interchanges
the tangent and the normal spaces. Moreover, since $J$ is
parallel, we deduce that
\begin{equation}\label{eqn:2.3}
\nabla^{\perp}_XJY=J\nabla_XY,\quad A_{JX}Y=-Jh(X,Y)=A_{JY}X.
\end{equation}
The last formula implies that the cubic form $g(h(X,Y),JZ)$ is
totally symmetric. The minimality condition on $M^n$ means that
${\it trace}\,h=0$, and one may notice that this is equivalent to
${\it trace}A_J=0$.

A straightforward computation shows that the equations of Gauss,
Codazzi and Ricci are
\begin{equation}\label{eqn:2.4}
R(X,Y)Z=\tilde{c}\,(\langle Y,Z \rangle X-\langle X,Z \rangle Y)+
[A_{JX},A_{JY}]Z,
\end{equation}
\begin{equation}\label{eqn:2.5}
(\nabla h)(X,Y,Z)=(\nabla h)(Y,X,Z),
\end{equation}
\begin{equation}\label{eqn:2.6}
R^\perp(X,Y)JZ=\tilde{c}( \langle Y,Z \rangle JX-\langle X,Z \rangle
JY )+J[A_{JX},A_{JY}]Z,
\end{equation}
where $X,Y,Z$ are tangent vector fields  and the covariant
derivative of $h$ is given by
\begin{equation}\label{eqn:2.7}
(\nabla h)(X,Y,Z)=\nabla^{\perp}_X(h(Y,Z))-h(\nabla_XY,Z)-h(Y,\nabla_XZ).
\end{equation}

Moreover, the following Ricci identity holds:
\begin{align}\label{eqn:2.8}
\begin{array}{l}
(\nabla^2 h)(X,Y,Z,W)-(\nabla^2 h)(Y,X,Z,W)\\[1mm]
\ \ \ \ =JR(X,Y)A_{JZ}W-h(R(X,Y)Z,W)-h(R(X,Y)W,Z),
\end{array}
\end{align}
where $X,Y,Z,W$ are tangent vector fields and
\begin{align}\label{eqn:2.9}
\begin{split}
(\nabla^2 h)(W, X,Y,Z)=&
\nabla^{\perp}_W(( \nabla h)(X,Y,Z))-(\nabla h)(\nabla _WX,Y,Z)\\
&-(\nabla h)(X,\nabla _WY,Z) -(\nabla h)(X,Y,\nabla _WZ).
\end{split}
\end{align}

In the following, we will prove an additional relation that is very
useful in our computations. To do so, we will make use of the
technique introduced in \cite{ALVW}, as the {\it Tsinghua
Principle}. First, take the covariant derivative in \eqref{eqn:2.5} with
respect to $W$, and use \eqref{eqn:2.9} and \eqref{eqn:2.5}, to obtain
straightforwardly that
\begin{equation}\label{eqn:2.10}
(\nabla^2h)(W,X,Y,Z)-(\nabla^2h)(W,Y,X,Z)=0.
\end{equation}

In the above equation we then cyclicly permute the first three
vector fields and express each time the left-hand side of the
equations using the Ricci identity in \eqref{eqn:2.8}. It then follows
that
\begin{equation}\label{eqn:2.11}
\begin{aligned}
    0=&R(W,X)Jh(Y,Z)-Jh(Y, R(W,X)Z)\\
    &+R(X,Y)Jh(W,Z)-Jh(W, R(X,Y)Z)\\
    &+R(Y,W)Jh(X,Z)-Jh(X, R(Y,W)Z).
\end{aligned}
\end{equation}

Furthermore, given [20, Corollary 58, p. 89], we know that
\begin{equation}\label{eqn:2.12}
R(X,Y)Z=c_1(\langle Y_1,Z_1 \rangle X_1-\langle X_1,Z_1
\rangle Y_1)+c_2(\langle Y_2,Z_2 \rangle X_2-\langle X_2,Z_2\rangle Y_2),
\end{equation}
where $X_i,Y_i,Z_i$ are the projections of $X,Y,Z$ on the
$TM_i^{n_i}$ component of $TM^n$, for $i=1,2$, respectively.

We recall the following useful definitions and theorems (see \cite{LW2}):
\begin{definition}\label{def1}
Let $\psi_i: (M_i^{n_i}, g_i)\to\mathbb{CP}^{n_i}(4)$, $i=1, 2$, be
two Lagrangian immersions and let $\tilde{\gamma}=(\tilde{\gamma}_1,
\tilde{\gamma}_2): I\to\mathbb{S}^3(1)\subset \mathbb{C}^2$ be a
Legendre curve. Then $\psi=\Pi(\tilde{\gamma}_1\tilde{\psi}_1;
\tilde{\gamma}_2\tilde{\psi}_2): I\times M_1^{n_1}\times M_2^{n_2} \to\mathbb{
CP}^n(4)$ is a Lagrangian immersion, where $n=n_1+n_2+1$,
$\tilde{\psi}_i: M_i^{n_i}\to\mathbb{S}^{2n_i+1}(1)$ are horizontal
lifts of $\psi_i$, $i=1, 2$, respectively, and $\Pi$ is the Hopf
fibration. We call $\psi$ a warped product Lagrangian immersion of
$\psi_1$ and $\psi_2$. When $n_1$ (or $n_2$) is zero, we call $\psi$
a warped product Lagrangian immersion of $\psi_2$ (or $\psi_1$) and
a point.
\end{definition}

\begin{definition}\label{def2}
In Definition \ref{def1}, when
\begin{equation}\label{eqn:2.13}
\tilde{\gamma}(t) = \Big( r_1 e^{i (\tfrac{r_2}{r_1}at)},r_2e^{i(-\tfrac{r_1}{r_2}at)}\Big),
\end{equation}
where $r_1, r_2$ and $a$ are positive constants with $r_1^2+r_2^2 = 1$, we call
$\psi$ a Calabi product Lagrangian immersion of $\psi_1$ and $\psi_2$.
When $n_1$ (or $n_2$) is zero, we call $\psi$ a Calabi product Lagrangian
immersion of $\psi_2$ (or $\psi_1$) and a point.
\end{definition}

\begin{them}[\cite{LW2}]\label{thm:2.1}
Let $\psi: M^n\to \mathbb{CP}^n(4)$ be a Lagrangian immersion. Then
$\psi$ is locally a Calabi product Lagrangian immersion of an
$(n-1)$-dimensional Lagrangian immersion $\psi_1 : M_1^{n-1}\to
\mathbb{CP}^{n-1}(4)$ and a point if and only if $M^n$ admits two
orthogonal distributions $\mathcal{D}_1$ (of dimension $1$, spanned
by a unit vector field $E_1$) and $\mathcal{D}_2$ (of dimension
$n-1$, spanned by $\{E_2,\ldots, E_n\}$) and there exist two real
constants $\lambda _1$ and $\lambda_2$ such that
\begin{equation}\label{eqn:2.14}
\begin{aligned}
h(E_1,E_1) &= \lambda_1 JE_1,\ h(E_1,E_i) =\lambda_2 JE_i, \ i= 2,\ldots,n,\\
\lambda_1 &\neq 2\lambda_2.
\end{aligned}
\end{equation}
Moreover, a Lagrangian immersion $\psi: M^n\to \mathbb{CP}^n(4)$,
satisfying the above conditions, has the following properties:
\begin{enumerate}
\item $\psi$ is Hamiltonian minimal if and only if $\psi_1$ is Hamiltonian minimal;
\item  $\psi$ is minimal if and only if $\lambda_2 =\pm\frac{1}{\sqrt{n}}$ and
$\psi_1$ is minimal. In this case, up to a reparametrization and a rigid motion
of $\mathbb{CP}^n$, locally we have $M^n=I\times M_1^{n-1}$ and $\psi$ is given by
$\psi = \Pi\circ \tilde{\psi}$ with
\begin{equation*}
\tilde{\psi}(t,p)=\Big(\sqrt{\tfrac{n}{n+1}}e^{i\frac{1}{n+1}t}
\tilde{\psi_1}(p), \sqrt{\tfrac{1}{n+1}}e^{-i\frac{n}{n+1}t}\Big),
\ (t,p)\in I\times M_1^{n-1},
\end{equation*}
where $\Pi$ is the Hopf fibration and $\tilde{\psi}_1 : M_1^{n-1}\to
S^{2n-1}(1)$ is the horizontal lift of $\psi_1$.
\end{enumerate}
\end{them}

\begin{them}[\cite{LW2}]\label{thm:2.2}
Let $\psi: M^n\rightarrow \mathbb{CP}^{n}(4)$ be a Lagrangian
immersion. If $M^n$ admits two orthogonal distributions
$\mathcal{D}_1$ (of dimension $1$, spanned by a unit vector field
$E_1$) and $\mathcal{D}_2$ (of dimension $n-1$, spanned by
$\{E_2,\ldots,E_n \}$), and that there exist local functions
$\lambda_1$, $\lambda_2$ such that \eqref{eqn:2.14} holds, then $M^n$ has
parallel second fundamental form if and only if $\psi$ is locally a
Calabi product Lagrangian immersion of a point and an
$(n-1)$-dimensional Lagrangian immersion $\psi_1:M_1^{n-1}\to
\mathbb{CP}^{n-1}(4)$, which has parallel second fundamental form.
\end{them}

\section{Proof of the Main Theorem}\label{sect:3}

In this section, we study a minimal Lagrangian isometric immersion
into a complex space form: $\psi :M^n\to\tilde M^n$, where
$M^n=M_1^{n_1}(c_1)\times
M_2^{n_2}(c_2)$, $n=n_1+n_2$ and $M_1^{n_1}(c_1)$ (resp.
$M_1^{n_2}(c_2)$) is $n_1$ (resp. $n_2$)-dimensional Riemannian
manifold with constant sectional curvature $c_1$ (resp. $c_2$). We
will prove the Main Theorem stated in introduction.

One should be aware that throughout the paper we will make the following
identifications. As $M^n=M_1^{n_1}\times M_2^{n_2}$, we can write a tangent vector field
$Z(p,q)= (X(p,q), Y(p,q))$ where $X(p,q) \in T_p M_1^{n_1}$ and $Y(p,q) \in T_q M_2^{n_2}$.
In general, the $X$ notation (as well as $X_i$, $1\leq i\leq n_1$) will
denote a vector tangent at $(p,q)\in M^n$, with zero components on $M_2^{n_2}$.
We will also identify $X(p,q) \in T_{p} M_1^{n_1}$ with $(X(p,q), 0) \in T_{(p,q)}(M_1^{n_1}\times M_2^{n_2})$
(and similarly $Y(p,q) \in T_{q} M_2^{n_2}$ with $(0, Y(p,q)) \in T_{(p,q)}(M_1^{n_1}\times M_2^{n_2})$.
Notice that, a priori, it means that $X$, as a vector field  depends on $q$
as well, not only on $p$.  One should have in mind this meaning when reading
$X\in T_pM_1^{n_1}$, respectively $Y\in T_qM_2^{n_2}$.
Nonetheless, a complete understanding will be acquired with the proofs of
Lemmas \ref{lem:3.6} and \ref{lem:3.7}, when we will actually see that due to
our particular choice of basis, $X$ only depends on $p$.

First of all, we consider the case $c_1^2+c_2^2\neq0$. We begin with
the following result.

\begin{lem}\label{lem:3.1}
If $c_1^2+c_2^2\neq0$, then the shape operator $A_J$ vanishes nowhere.
\end{lem}
\begin{proof}
Assume that $A_J$ vanishes at the point $p \in M^n$. From equation
\eqref{eqn:2.4} it follows that $R(X,Y)Z=\tilde{c}(\langle Y,Z\rangle
X-\langle X,Z\rangle Y)$, which yields that $M^n$ has constant
sectional curvature $\tilde{c}$ at $p$. Moreover, by taking
$X_1,X_2,X_2$ in \eqref{eqn:2.4} and \eqref{eqn:2.12}, we obtain that
$c_1=\tilde{c}$ and then by taking $X,Y,Y$ in \eqref{eqn:2.4} and
\eqref{eqn:2.12}, $X\in T_pM_1^{n_1}$, $Y\in T_pM_2^{n_2}$, we get $\tilde{c}=0$.
Similarly, taking $Y_1,Y_2,Y_2\in T_pM_2^{n_2}$ in \eqref{eqn:2.4} and
\eqref{eqn:2.12}, we get that $c_2=0$. Therefore, we get a
contradiction with $c_1^2+c_2^2\neq0$.
\end{proof}

For $c_1^2+c_2^2\neq0$, if $c_1c_2=0$, without loss of generality,
we may assume that $c_1=0$ and $c_2\neq0$. Therefore, we are left to
consider the following two cases:

\vskip 1mm

\textbf{Case (i)}: $c_1=0$ and $c_2\neq 0$;\ \ \textbf{Case (ii)}:
$c_1\neq 0$ and $c_2\neq 0$.

\vskip 2mm

\subsection{}\label{sect:3.1}~

In this subsection, we will deal with {\bf Case (i)} and prove the
following result.
\begin{them}\label{thm:3.1}
Let $\psi :M^n\to \tilde M^n(4\tilde{c})$ be a minimal Lagrangian
isometric immersion into a complex space form such that
$(M^n,\langle\cdot,\cdot\rangle)=M_1^{n_1}(c_1)\times
M_2^{n_2}(c_2)$ and Case (i) occurs. Then we have
$c_2=\tfrac{n_1+n_2+1}{n_2+1}\tilde{c}>0$, say $\tilde{c}=1$, so the
ambient space is $\mathbb{ CP}^n(4)$ and the immersion is congruent
with
\begin{equation*}
\tfrac{1}{n+1}(e^{iu_1},\ldots,e^{iu_{n_1}}, a e^{i u_{n_1+1}}y_1,
\ldots,a e^{i u_{n_1+1}}y_{n_2+1}),
\end{equation*}
where
\begin{enumerate}
\item $(y_1,y_2,\ldots,y_{n_2+1})$ describes the standard sphere
$\mathbb{S}^{n_2}\hookrightarrow \mathbb{R}^{n_2+1}\hookrightarrow
\mathbb{C}^{n_2+1}$,

\item $a=\sqrt{n_2+1}$,

\item $u_1+\cdots+u_{n_1}+a^2 u_{n_1+1}=0$.
\end{enumerate}
\end{them}

The proof of Theorem \ref{thm:3.1} consists of several lemmas as
following.

\begin{lem}\label{lem:3.2}
Let $\{X_i\}_{1\le i\le n_1}$ and $\{Y_j\}_{1\le j\le n_2}$ be
orthonormal bases of $M_1^{n_1}(c_1)$ and $M_2^{n_2}(c_2)$,
respectively. Then we have
\begin{equation}\label{eqn:3.1}
\langle A_{JX_i} X_j, Y_k\rangle=0,
\end{equation}
and
\begin{equation} \label{eqn:3.2}
\langle A_{JX_i}Y_j , Y_k\rangle=\left\{
\begin{aligned}
&0,\ {\rm if }\ j\neq k,\\
&\mu (X_i), \ {\rm if }\ j=k,
\end{aligned}
\right.
\end{equation}
where $\mu(X_i)=:\mu_i$ depends only on $X_i$ for each $i=1,\ldots,n_1$.
\end{lem}

\begin{proof}
Expressing \eqref{eqn:2.11} for $X=Y_k, Y=Y_l, Z=X_i, W=X_j$, $k\neq l$,
and using \eqref{eqn:2.12}, we see that there is only one term
remaining in the right-hand side: $0=R(Y_k,Y_l)A_{JX_i} X_j$. Using
\eqref{eqn:2.12} again, we get
\begin{equation}\label{eqn:3.3}
0=\langle Y_l, A_{JX_i} X_j \rangle Y_k-\langle Y_k, A_{JX_i} X_j
\rangle Y_l.
\end{equation}
It follows immediately the assertion \eqref{eqn:3.1} that
\begin{equation}\label{eqn:3.4}
\langle Y_l, A_{JX_i} X_j\rangle=0,\ \ 1\leq i,j\leq n_1, \ 1\leq l\leq n_2.
\end{equation}

For the second relation, we proceed similarly by choosing in
\eqref{eqn:2.11}: $X=Y_m$, $Y=X_i$, $Z=Y_l$, $W=Y_k$, we obtain
\begin{equation}\label{eqn:3.5}
0=-c_2(\langle A_{JX_i} Y_l,Y_m\rangle Y_k-\langle A_{JX_i} Y_l,
Y_k\rangle Y_m-\delta_{ml}A_{JX_i} Y_k+\delta_{kl}A_{JX_i} Y_m).
\end{equation}

In \eqref{eqn:3.5}, let $k,l,m$ be distinct, then we get
\begin{equation}\label{eqn:3.6}
\langle A_{JX_i} Y_l,Y_m\rangle=0,\ 1\leq i\leq n_1,\ 1\leq l, m\leq
n_2,\ l\not=m.
\end{equation}

Again in \eqref{eqn:3.5}, let assume that $l=m\neq k$, then we
have
\begin{equation}\label{eqn:3.7}
\langle A_{JX_i} Y_l,Y_l\rangle=\langle A_{JX_i} Y_k,Y_k\rangle, \
1\leq i\leq n_1,\ 1\leq l,k\leq n_2,\ l\neq k.
\end{equation}

By \eqref{eqn:3.4}, \eqref{eqn:3.6} and \eqref{eqn:3.7},
we have $\mu(X_i)$ depends only on $X_i$ such that
\begin{equation*}
 A_{JX_i} Y_l=  \mu (X_i)Y_l,\ 1\leq i\leq n_1,\ 1\leq l\leq n_2.
\end{equation*}

Then the assertion \eqref{eqn:3.2} immediately follows.
\end{proof}

\begin{lem}\label{lem:3.3}
Let $\{X_i\}_{1\le i\le n_1}$ be an orthonormal basis in the tangent
space of $M_1^{n_1}$ at a point. Then it holds that
\begin{equation}\label{eqn:3.8}
\mu(X_1)^2+\cdots+\mu(X_{n_1})^2=\tfrac{n_1}{n_2+1}\tilde c.
\end{equation}
\end{lem}
\begin{proof}
We compute the sectional curvature $K(\pi(X_i,Y_j))$ of the plane $\pi$
spanned by $X_i$ and $Y_j$, for some fixed $i=1,\ldots,n_1$ and some fixed
$j=1,\ldots,n_2$. We use on the one hand \eqref{eqn:2.12} and on the other
hand \eqref{eqn:2.4} together with \eqref{eqn:3.2} to obtain
\begin{align*}
0=&\tilde c + \langle A_{JY_j}Y_j, A_{JX_i}X_i \rangle-
\langle A_{JX_i}Y_j, A_{JY_j}X_i \rangle\\
=&\tilde c - \mu(X_i)^2+\langle A_{JY_j}Y_j, A_{JX_i}X_i \rangle,\ \
1\le i\le n_1,\ 1\le j\le n_2.
\end{align*}
Taking summation over $i=1,\ldots,n_1$, and using Lemma \ref{lem:3.2},
we get
\begin{equation}\label{eqn:3.9}
\begin{aligned}
0&=n_1\tilde c -\sum_{i=1}^{n_1}\mu(X_i)^2+\langle A_{JY_j} Y_j, \sum_{i=1}^{n_1}A_{JX_i}X_i \rangle\\
&=n_1\tilde c
-\sum_{i=1}^{n_1}\mu(X_i)^2+\sum_{k=1}^{n_1}\sum\limits_{i=1}^{n_1}
\langle A_{JX_k} X_i, X_i \rangle \mu(X_k).
\end{aligned}
\end{equation}
However, the minimality condition implies that for each $k=1,\ldots,n_1$ we have
\begin{equation}\label{eqn:3.10}
0=\sum_{i=1}^{n_1} \langle A_{JX_k} X_i, X_i \rangle +
\sum_{j=1}^{n_2} \langle A_{JX_k} Y_j, Y_j \rangle=\sum_{i=1}^{n_1}
\langle A_{JX_k} X_i, X_i\rangle+n_2\mu(X_k).
\end{equation}

Therefore, from \eqref{eqn:3.9} and \eqref{eqn:3.10}, we obtain
\begin{equation}\label{eqn:3.11}
\mu(X_1)^2+\cdots+\mu(X_{n_1})^2=\tfrac{n_1}{n_2+1}\tilde c.
\end{equation}

This completes the proof of Lemma \ref{lem:3.3}.
\end{proof}

Next, we will describe the construction of a local frame of vector
fields for which we can determine the values of the  shape operator
$A_{J}$. This is a crucial step and will be stated in Lemma
\ref{lem:3.5}. Let us describe first a general method for choosing
suitable orthonormal vectors at a point on $M^n$, which will be used
recurrently in the proof of Lemma \ref{lem:3.5}. The main idea
originates from the very similar situation in studying affine
hyperspheres in \cite{CHMV,HLV,VLS}.

Let $(p,q) \in M^n$ and $U_pM_1^{n_1}=\{ u\in T_pM_1^{n_1}| \langle
u,u\rangle=1\}$. As the metric on $M_1^{n_1}$ is positive definite,
we have that $U_pM_1^{n_1}$ is compact. We define on this set the
functions
\begin{equation}\label{eqn:3.12}
f_{(p,q)}(u)=\langle A_{Ju}u,u\rangle,\ u\in U_pM_1^{n_1}.
\end{equation}
We know that there exists $e_1\in U_pM_1^{n_1}$ for which $f_{(p,q)}$
attains an absolute maximum: $f_{(p,q)}(e_1)= \langle
A_{Je_1}e_1,e_1\rangle=:\lambda_1 $. Let $u\in U_pM_1^{n_1}$ such
that $\langle u, e_1\rangle=0$ and define $g(t)= f_{(p,q)}(\cos(t)
e_1+\sin(t) u )$. One may check that
\begin{align}
g'(0)&=3 \langle A_{Je_1}e_1,u\rangle,\label{eqn:3.13}\\
g''(0)&=6 \langle A_{Je_1}u,u\rangle- 3f_{(p,q)}(e_1).\label{eqn:3.14}
\end{align}

Since $g$ attains an absolute maximum for $t=0$,
we have that $g'(0)=0$ and $g''(0)\leq 0$, i.e.
\begin{equation}\label{eqn:3.15}
\left\{
\begin{aligned}
&\langle A_{Je_1}e_1,u \rangle=0,\\
&\langle A_{Je_1}e_1,e_1\rangle\geq2\langle A_{Je_1}u,u
\rangle,\ u\perp e_1,\ \langle u,u \rangle=1.
\end{aligned}
\right.
\end{equation}
Therefore, $e_1$ is an eigenvector of $ A_{Je_1} $ with $\lambda_1$
the corresponding eigenvalue. Since $A_{Je_1}$ is self-adjoint, we
can further choose orthonormal vectors $e_2, \ldots, e_{n_1}$, which
are eigenvectors of $A_{Je_1}$, with respectively the eigenvalues
$\lambda_{2},\ldots, \lambda_{n_1}$. To sum up, we have
\begin{equation}\label{eqn:3.16}
A_{Je_1}e_i=\lambda_{i} e_i,\ i=1,\ldots, n_1;\ \ \lambda_{1} \geq
2\lambda_{i}\ \ {\rm for}\ \ i\ge2.
\end{equation}

\begin{lem}\label{lem:3.4}
Let $(p,q)\in M_1^{n_1}\times M_2^{n_2}$ and $\{X_i\}_{1\le i\le
n_1}$ and $\{Y_j\}_{1\le j\le n_2}$ be arbitrary orthonormal bases of
$T_pM_1^{n_1}$ and $T_qM_2^{n_2}$, 
respectively. Then
\begin{equation}\label{eqn:3.17}
A_{JY_j}Y_k=(\mu_1 X_1+\cdots+\mu_{n_1} X_{n_1})\delta_{jk},\ \ 1\le
j,k\le n_2,
\end{equation}
where $\mu_i:=\mu(X_i)$ with $\mu$ defined as before. Moreover, we have
$c_2=\frac{n_1+n_2+1}{n_2+1}\tilde{c}$.
\end{lem}

\begin{proof}
From Lemma \ref{lem:3.2} we know that
\begin{equation*}
A_{JY_j}Y_k=(\mu_1
X_1+\cdots+\mu_{n_1}X_{n_1})\delta_{jk}+\sum_{l=1}^{n_2}\alpha_l^{jk}Y_l,
\end{equation*}
for real numbers $\alpha_1^{jk},\ldots,\alpha_{n_2}^{jk}$.

Now, we claim that $\alpha_l^{jk}=0$ for all possible indexes, or
equivalently,
\begin{equation}\label{eqn:3.18}
\langle A_{JY_j}Y_k, Y_l\rangle=0,\ {\rm for\ any}\ Y_j,Y_k,Y_l\in
T_qM_2^{n_2}.
\end{equation}

We will verify the claim by contradiction.

In fact, if it did not hold, then we could choose a unit vector
$Y_1(p,q)\in U_qM_2^{n_2}$ such that $\alpha_1:=\langle A_{J
Y_1}Y_1,Y_1\rangle>0$ is the maximum of the function $f_{(p,q)}$ defined
on $U_qM_2^{n_2}$.

Define an operator $\mathcal{A}$ on $T_qM_2^{n_2}$ by
$$
\mathcal{A}(Y)=A_{J Y_1}Y-\langle A_{J Y_1}Y,X_1\rangle X_1
-\cdots-\langle A_{J Y_1}Y,X_{n_1}\rangle X_{n_1}.
$$

It is easy to show that $\mathcal{A}$ is self-adjoint and $Y_1$ is
one of its eigenvectors. We can choose orthonormal vectors
$Y_2,\ldots,Y_{n_2}\in U_qM_2^{n_2}$ orthogonal to $Y_1$, which are
the remaining eigenvectors of the operator $\mathcal{A}$, associated
to the eigenvalues $\alpha_2,\ldots,\alpha_{n_2}$ (notice that we have
changed the notation for the corresponding $\alpha_l^{jk}$ for more
simplicity). Therefore, we have
\begin{equation}\label{eqn:3.19}
\left\{
\begin{aligned}
&A_{J Y_1}Y_1=\mu_1X_1+\cdots+\mu_{n_1}X_{n_1}+\alpha_1Y_1,\\
& A_{J Y_1}Y_i=\alpha_iY_i,\ 1< i\leq n_2.
\end{aligned}
\right.
\end{equation}

Taking in \eqref{eqn:2.4} $X=Z=Y_1,Y=Y_i,1<i\leq n_2$, using
\eqref{eqn:3.19} and Lemmas \ref{lem:3.2} and \ref{lem:3.3}, we can
obtain
\begin{equation}\label{eqn:3.20}
\alpha_i^2-\alpha_1\alpha_i-\tfrac{n_1+n_2+1}{n_2+1}\tilde{c}+c_2=0.
\end{equation}
It follows that there exist an integer $n_{2,1}$, $0\leq n_{2,1}\leq
n_2-1$, if necessary after renumbering the basis, such that
\begin{equation}\label{eqn:3.21}
\left\{
\begin{aligned}
&\alpha_2=\cdots=\alpha_{n_{2,1}+1}=\tfrac12\Big(\alpha_1+\sqrt{\alpha_1^2
+4(\tfrac{n_1+n_2+1}{n_2+1}\tilde{c}-c_2)}\,\Big),\\
&\alpha_{n_{2,1}+2}=\cdots=\alpha_{n_2}=\tfrac12\Big(\alpha_1-\sqrt{\alpha_1^2
+4(\tfrac{n_1+n_2+1}{n_2+1}\tilde{c}-c_2)}\,\Big).
\end{aligned}
\right.
\end{equation}

Using Lemma \ref{lem:3.2}, \eqref{eqn:3.19}, \eqref{eqn:3.21} and
${\rm trace}\,A_{J Y_1}=0$, we have
\begin{equation}\label{eqn:3.22}
\alpha_1=\sqrt{\tfrac{4(\frac{n_1+n_2+1}{n_2+1}\tilde{c}
-c_2)}{\big(\tfrac{n_2+1}{n_2-2n_{2,1}-1}\big)^2-1}}.
\end{equation}

Therefore, if there exists a unit vector field $ V\in TM_2^{n_2}$ such
that $A_{JV}V=\lambda V+\mu_1X_1+\cdots+\mu_{n_1}X_{n_1}$, then we
see that
\begin{equation}\label{eqn:3.23}
\lambda\in\left\{\sqrt{\tfrac{4\big(\frac{n_1+n_2+1}{n_2+1}\tilde{c}
-c_2\big)}{\big(\tfrac{n_2+1}{n_2-2n_{2,1}-1}\big)^2-1}}\,\right\}_{0\le
n_{2,1}\le n_2-1}.
\end{equation}

Moreover, $\alpha_1$ is the absolute maximum of $f_{(p,q)}$ if and only if
\begin{equation}\label{eqn:3.24}
\alpha_1=\sqrt{\tfrac{4\big(\tfrac{n_1
+n_2+1}{n_2+1}\tilde{c}-c_2\big)}{\big(\tfrac{n_2+1}{n_2-1}\big)^2-1}},\
\ {\rm corresponding\ to} \ n_{2,1}=0.
\end{equation}

Next, we show that if $f_{(p,q)}$ attains an absolute maximum in $Y_1$, we
can extend $Y_1$ differentiably to a unit vector field which is also
denoted by $Y_1$ on a neighbourhood $U$ of $(p,q)$ such that, at every
point $(p',q')\in U$, $f_{(p',q')}$ attains an absolute maximum in $Y_1(p',q')$.

In order to achieve that purpose, let $\{E_1,\ldots,E_{n_2}\}$ be an
arbitrary differentiable orthonormal basis defined on a
neighbourhood $U'$ of $(p,q)$ such that $E_1(p,q)=Y_1$. Then, we define a
function $\gamma$ by
$$
\gamma:\mathbb{R}^{n_2}\times U'\rightarrow \mathbb{R}^{n_2}:(a_1,\ldots,a_{n_2},(p',q'))\mapsto(b_1,\ldots,b_{n_2}),
$$
$$
b_k=\sum_{i,j=1}^{n_1}a_ia_j\langle
A_{JE_i}E_j,E_k\rangle-\alpha_1a_k,\ 1\le k\le n_2.
$$
Using the fact that $f_{(p,q)}$ attains an absolute maximum in $E_1(p,q)$,
we then obtain that
\begin{equation*}
\begin{aligned}
\tfrac{\partial b_k}{\partial
a_m}(1,0,\ldots,0,{(p,q)})&=2\langle(A_{JE_1{(p,q)}}
E_m{(p,q)},E_k{(p,q)}\rangle-\alpha_1\delta_{km}\\
&=\left\{
\begin{aligned}
&0,\ \ {\rm if}\  k\neq m, \\
&\alpha_1,\  \ {\rm if}\ k=m=1,\\
&2\alpha_k-\alpha_1,\ \ {\rm if}\ k= m>1.
\end{aligned}
\right.
\end{aligned}
\end{equation*}

Since $\alpha_1>0$ and given \eqref{eqn:3.21}, we have
$2\alpha_k-\alpha_1\neq0$ for $k\ge2$. Hence the implicit function
theorem shows that there exist differentiable functions $a_1,\ldots,
a_{n_2}$, defined on a neighbourhood $U$ of $(p,q)$, such that
$$
a_1(p,q)=1,\ a_2(p,q)=0,\ \ldots,\ a_{n_2}(p,q)=0.
$$
Define the local vector field $V$ by
$$
V=a_1E_1+\cdots+a_{n_1}E_{n_1}.
$$
Then we have $V(p,q)=Y_1$ and $A_{JV}V=\alpha_1 V+\mu_1\langle
V,V\rangle X_1 +\cdots+\mu_{n_1}\langle V,V\rangle X_{n_1}$.
Hence
$$
A_{J\tfrac{V}{\sqrt{\langle V,V\rangle}}}\tfrac{V}{\sqrt{\langle V,V\rangle}}
=\tfrac{\alpha_1}{\sqrt{\langle V,V\rangle}}\tfrac{V}{\sqrt{\langle V,V\rangle}}
+\mu_1X_1+\cdots+\mu_{n_1}X_{n_1}.
$$
By \eqref{eqn:3.23}, the continuity of $\tfrac{\alpha_1}{\sqrt{\langle V,V\rangle}}$
and $\langle V,V\rangle(q)=1$, we can derive that $\langle V,V\rangle=1$
identically. Therefore, for any point $(p',q')\in U$, $f_{('p,q')}$ attains an
absolute maximum at $V(p',q')$. Let $Y_1=V$ and take orthonormal vector
fields $Y_2,\ldots,Y_{n_2}$ orthogonal to $Y_1$, then
$\{Y_1,\ldots,Y_{n_1}\}$ is a local basis satisfying
\begin{equation}\label{eqn:3.25}
\left\{
\begin{aligned}
&A_{JY_1}Y_1=\mu_1X_1+\cdots+\mu_{n_1}X_{n_1}+\alpha_1Y_1,\\& A_{JY_1}Y_i=\alpha_iY_i,\ \ 1< i\leq n_2,
\end{aligned}
\right.
\end{equation}
where, $\alpha_1$ is defined by \eqref{eqn:3.24}, and
\begin{equation}\label{eqn:3.26}
\alpha_2= \cdots=\alpha_{n_2}=\tfrac12\big(\alpha_1-\sqrt{\alpha_1^2
+4(\tfrac{n_1+n_2+1}{n_2+1}\tilde{c}-c_2)}\,\big).
\end{equation}

We recall that on the product manifold $M^n$ we know that $\langle \nabla_{Y_i}Y_j,X\rangle =0$,
for $i,j=1,\ldots,n_2$ and $X$ tangent to $M_1$. Applying \eqref{eqn:2.5},
and \eqref{eqn:3.24}--\eqref{eqn:3.26}, we have that
\begin{equation}\label{eqn:3.27}
\nabla_{Y_i}Y_1=0,\ \ 1\leq i\leq n_2.
\end{equation}

Hence, we have $R(Y_1,Y_2)Y_1=0$, a contradiction to the fact that
$c_2\neq0$. This verifies the claim and thus \eqref{eqn:3.17} follows.
Moreover, using \eqref{eqn:2.4}, \eqref{eqn:2.12} and \eqref{eqn:3.17}, we
easily get the relation $c_2=\frac{n_1+n_2+1}{n_2+1}\tilde{c}$.
\end{proof}

\begin{lem}\label{lem:3.5}
In Case (i), we have $\tilde{c}>0$. Moreover, there exist local orthonormal
frames of vector fields $\{X_i\}_{1\le i\le n_1}$ of $M_1^{n_1}$ and
$\{Y_j\}_{1\le j\le n_2}$ of $M_2^{n_2}$, respectively, such that
the operator $A_J$ takes the following form:
\begin{equation}\label{eqn:3.28}
\left\{
\begin{aligned}
A_{JX_1}X_1&=\lambda_{1,1}X_1,\\
A_{JX_i}X_i&=\mu_1 X_1+\cdots+\mu_{i-1} X_{i-1}+\lambda_{i,i}X_i,\ 1< i\leq n_1,\\
A_{JX_i}X_j&=\mu_i X_j,\ 1\leq i<j,\\
A_{JX_i}Y_j&=\mu_i Y_j,\ 1\leq i\leq n_1,\ 1\leq j\leq n_2,
\end{aligned}
\right.
\end{equation}
where $\lambda_{i,i},\, \mu_i$ are constants and satisfy
\begin{equation}\label{eqn:3.29}
\lambda_{i,i}+(n-i)\mu_i=0,\ \ 1\le i\le n_1.
\end{equation}
\end{lem}

\begin{proof}
We will give the proof by induction on the index $i$ of
$A_{JX_i}$. According to general principles, this consists of two
steps as below.

\vskip 2mm

\noindent {\bf The first step of induction.}

\vskip 1mm
In this step, we should verify the assertion for $i=1$. To do so, we
have to show that, around any given $(p,q)\in M_1^{n_1}\times
M_2^{n_2}$, there exist an orthonormal frame of vector fields
$\{X_i\}_{1\leq i\leq n_1}$ of $TM_1^{n_1}$, $\{Y_j\}_{1\leq i\leq
n_2}$ of $TM_2^{n_2}$, and smooth functions $\lambda_{1,1}$ and
$\mu_1$, so that we have
$$
\left\{
\begin{aligned}
&A_{JX_1}X_1=\lambda_{1,1} X_1,\ \ A_{JX_1}Y_j=\mu_1  Y_j,\ \ 1\le
j\le n_2,\\
&A_{JX_1}X_i=\mu_1 X_i,\ \ 2\leq i\leq n_1,\\
&\lambda_{1,1}+(n-1)\mu_1=0.
\end{aligned}
\right.
$$
The proof of the above conclusion will be divided into four claims
as below.

\vskip 2mm

\noindent\textbf{Claim I-(1)}. {\it Given $(p,q)\in M_1^{n_1}\times
M_2^{n_2}$, there exist orthonormal bases $\{X_i\}_{1\leq i\leq
n_1}$ of $T_pM_1^{n_1}$, $\{Y_j\}_{1\leq i\leq n_2}$ of
$T_qM_2^{n_2}$, and real numbers $\lambda_{1,1}>0$,
$\lambda_{1,2}=\cdots=\lambda_{1,n_1}$ and $\mu_1$, such that the
following relations hold:
$$
\left\{
\begin{aligned}
A_{JX_1}X_1&=\lambda_{1,1} X_1,\ \ A_{JX_1}X_i=\lambda_{1,i} X_i,\
2\le i\le n_1,\\
A_{JX_1}Y_j&=\mu_1 Y_j,\ 1\le j\le n_2.
\end{aligned}
\right.
$$
Moreover, $\lambda_{1,1}$ is the maximum of $f_{(p,q)}$ defined on
$U_{p}M_1^{n_1}$. In particular, $\tilde{c}>0$.}

\begin{proof}[Proof of Claim I-(1)]

First, if for an orthonormal basis $\{X_i\}_{1\leq i\leq n_1}$ and
for any $i,j,k=1,\ldots, n_1$,  $\langle
A_{JX_i}X_j,X_k\rangle=0$  holds, then by the fact $trace A_{JX_i}=0$ and
Lemma \ref{lem:3.2}, we get $\mu_i=0$. This further implies by Lemma
\ref{lem:3.3} that $\tilde{c}=0$. From this, using \eqref{eqn:2.4},
\eqref{eqn:2.12} and Lemma \ref{lem:3.4}, we can compute the sectional
curvature of the section spanned by $Y_1$ and $Y_2$ to obtain that
$c_2=0$, which is a contradiction.

Accordingly, following the idea  described right before Lemma
\ref{lem:3.4}, we can choose a vector $X_1 \in U_pM_1^{n_1}$ such
that $f_{(p,q)}$ on $U_pM_1^{n_1}$ attains its absolute maximum
$\lambda_{1,1}>0$ at $X_1$. Then, we
can choose an orthonormal basis $\{X_i\}_{1\le i\le n_1}$ of
$T_pM_1^{n_1}$ and an arbitrary orthonormal basis $\{Y_j\}_{1\leq
i\leq n_2}$ of $T_qM_2^{n_2}$, such that, for $2\le k\le n_1$,
$A_{JX_1}X_k=\lambda_{1,k} X_k$ and
$\lambda_{1,1}\ge2\lambda_{1,k}$. Moreover, by Lemma \ref{lem:3.2},
$A_{JX_1}Y_j=\mu_1 Y_j$ for $1\le j\le n_2$.

Next, we will show that $\lambda_{1,2}=\cdots=\lambda_{1,n_1}$, and
that $\lambda_{1,1},\lambda_{1,2}$ and $\mu_1$ are all constants
independent of $(p,q)$.

Taking in \eqref{eqn:2.4} that $X=Z=X_1$ and $Y=X_k$ for $k\ge2$, and
using \eqref{eqn:2.12}, we obtain
\begin{equation}\label{eqn:3.30}
\lambda_{1,k}^2-\lambda_{1,1}\lambda_{1,k}-\tilde{c}=0,\ \ 2\leq k\leq n_1.
\end{equation}
As $\tilde{c}\ge0$  by \eqref{eqn:3.11} and
$\lambda_{1,1}\ge2\lambda_{1,k}$ for $2\le k\le n_1$, then
\eqref{eqn:3.30} implies that
\begin{equation}\label{eqn:3.31}
\lambda_{1,2}=\cdots=\lambda_{1,n_1}=\frac{1}{2}\Big(\lambda_{1,1}-\sqrt{\lambda_{1,1}^2+4\tilde{c}}\,\Big).
\end{equation}

Similarly, taking $X=Z=X_1$ and $Y\in U_qM_2^{n_2}$ in \eqref{eqn:2.4}
and using \eqref{eqn:2.12} and Lemma \ref{lem:3.2}, we get
\begin{equation}\label{eqn:3.32}
\mu_1^2-{\mu_1}\lambda_{1,1}-\tilde{c}=0.
\end{equation}
Thus we obtain
\begin{equation}\label{eqn:3.33}
\mu_1
=\frac{1}{2}\Big(\lambda_{1,1}+\varepsilon_1\sqrt{\lambda_{1,1}^2+4\tilde{c}}\Big),\
\varepsilon_1=\pm 1.
\end{equation}
Then, applying ${\it trace}\,A_{JX_1}=0$, we get
\begin{equation}\label{eqn:3.34}
\frac{1}{2}(n+1)\lambda_{1,1}+\frac{1}{2}(\varepsilon_1n_2-n_1+1)\sqrt{\lambda_{1,1}^2+4\tilde{c}}=0.
\end{equation}
It follows that $\varepsilon_1n_2-n_1+1\not=0$ and
\begin{equation}\label{eqn:3.35}
\Big[\Big(\frac{n+1}{\varepsilon_1n_2-n_1+1}\Big)^2-1\Big]\lambda_{1,1}^2=4\tilde{c}.
\end{equation}
Moreover, \eqref{eqn:3.35} shows that $\tilde{c}>0$, and that
\begin{equation}\label{eqn:3.36}
\lambda_{1,1}=2\sqrt{\tfrac{\tilde{c}}{(\tfrac{n+1}{\varepsilon_1n_2-n_1+1})^2-1}}.
\end{equation}
This, together with \eqref{eqn:3.33}, implies that $\lambda_{1,1}$,
$\lambda_{1,2}=\cdots=\lambda_{1,n_1}$ and $\mu_1$ are all constants
independent of $(p,q)$.
\end{proof}

\noindent\textbf{Claim I-(2)}. {\it
$\lambda_{1,2}=\cdots=\lambda_{1,n_1}=\mu_1$ and
$\lambda_{1,1}+(n-1)\mu_1=0$.}

\begin{proof}[Proof of Claim I-(2)]

From \eqref{eqn:3.31} and \eqref{eqn:3.33}, the first assertion is
equivalent to showing that $\varepsilon_1=-1$. Suppose on the
contrary that $\varepsilon_1=1$. Then we have
\begin{equation}\label{eqn:3.37}
\mu_1\lambda_{1,2}=-\tilde{c}.
\end{equation}
Corresponding to the case $c_2\neq0$ we have $n_2\geq2$, then
\eqref{eqn:3.34} implies that
\begin{equation}\label{eqn:3.38}
n_1>n_2+1\geq3.
\end{equation}
We rechoose a vector $X_2\in U_pM_1^{n_1}$, which is orthogonal to
$X_1$ and such that $\lambda_{2,2}=\langle A_{JX_2}X_2,X_2\rangle$
is the maximum of $f_{(p,q)}$ on $\{u\in U_pM_1^{n_1}\,|\,u\perp X_1\}$.

Define $\mathcal{A}$ on $\{u\in T_pM_1^{n_1}\,|\,u\perp X_1\}$ by
$\mathcal{A}(X)=A_{JX_2}X-\langle A_{JX_2}X,X_1\rangle X_1$. It is
easy to show that $\mathcal{A}$ is self-adjoint and $X_2$ is one of
its eigenvectors. We can choose an orthonormal basis
$\{X_3,\ldots,X_{n_1}\}$ for $\{u\in T_pM_1^{n_1}\,|\,u\perp X_1,
u\perp X_2\}$ so that they are the remaining eigenvectors of the
operator $\mathcal{A}$, associated to eigenvalues
$\lambda_{2,3},\ldots,\lambda_{2, n_1}$. In this way, we have
obtained
\begin{equation}\label{eqn:3.39}
A_{JX_2}X_2=\lambda_{1,2}X_1+\lambda_{2,2}X_2,\
A_{JX_2}X_k=\lambda_{2,k}X_k,\ 3\leq k\leq n_1.
\end{equation}
Taking $X=Z=X_2,Y=X_k$ in \eqref{eqn:2.4} and using \eqref{eqn:3.39}
together with \eqref{eqn:2.12}, we obtain
\begin{equation}\label{eqn:3.40}
\lambda_{2,k}^2-\lambda_{2,2}\lambda_{2,k}-\tilde{c}-\lambda_{1,2}^2=0,\ 3\leq k\leq n_1.
\end{equation}
Given that $\lambda_{2,2}\ge 2\lambda_{2,k}$, this implies that
\begin{equation}\label{eqn:3.41}
\lambda_{2,k}=\frac{1}{2}\Big(\lambda_{2,2}-\sqrt{\lambda_{2,2}^2
+4(\tilde{c}+\lambda_{1,2}^2)}\Big),\ 3\le k\leq n_1.
\end{equation}

Similarly, taking $X=Z=X_2$ and $Y\in
U_qM_2^{n_2}$  in \eqref{eqn:2.4} and using \eqref{eqn:3.39} and \eqref{eqn:2.12}, we get
\begin{equation}\label{eqn:3.42}
{\mu_2}^2-{\mu_2}\lambda_{2,2}-\tilde{c}-\mu_1\lambda_{1,2}=0.
\end{equation}
Combining \eqref{eqn:3.37} with \eqref{eqn:3.42} we get
\begin{equation}\label{eqn:3.43}
\mu_2^2-{\mu_2}\lambda_{2,2}=0.
\end{equation}
Therefore, we have
\begin{equation}\label{eqn:3.44}
\mu_2=\frac{1}{2}(\lambda_{2,2}+\varepsilon_2\lambda_{2,2}),\ \varepsilon_2=\pm1.
\end{equation}
By using \eqref{eqn:3.39}, \eqref{eqn:3.41}, \eqref{eqn:3.44} and ${\it trace}\,A_{JX_2}=0$, we have
\begin{equation}\label{eqn:3.45}
\lambda_{2,2}+\frac{1}{2}(n_1-2)\Big(\lambda_{2,2}-\sqrt{\lambda_{2,2}^2
+4(\tilde{c}+\lambda_{1,2}^2)}\,\Big)+\frac12{n_2}(\lambda_{2,2}+\varepsilon_2\lambda_{2,2})=0.
\end{equation}
Hence we have
\begin{equation}\label{eqn:3.46}
\lambda_{2,2}=2\sqrt{\tfrac{\tilde{c}+\lambda_{1,2}^2}{\big(\tfrac{n_1+n_2+\varepsilon_2n_2}{n_1-2}\big)^2-1}}.
\end{equation}
Note that for $\varepsilon_1=1$, \eqref{eqn:3.36} gives
\begin{equation}\label{eqn:3.47}
\lambda_{1,1}=2\sqrt{\tfrac{\tilde{c}}{\big(\tfrac{n_1+n_2+1}{n_1-n_2-1}\big)^2-1}}.
\end{equation}
Using \eqref{eqn:3.38}, we have
\begin{equation*}
\begin{aligned}
\tfrac{n_1+n_2+1}{n_1-n_2-1}-\tfrac{n_1+n_2+\varepsilon_2n_2}{n_1-2}\geq&
\tfrac{n_1+n_2+1}{n_1-n_2-1}-\tfrac{n_1+2n_2}{n_1-2}\\
=&\tfrac{n_1-n_2-1+2(n_2+1)}{n_1-n_2-1}-\tfrac{n_1-2+2n_2+2}{n_1-2}\\
=&\tfrac{2(n_2+1)(n_2-1)}{(n_1-n_2-1)(n_1-2)}>0.
\end{aligned}
\end{equation*}
It follows that $\lambda_{2,2}>\lambda_{1,1}$. This is a
contradiction.

We have proved that  $\varepsilon_1=-1$ and thus
$\lambda_{1,2}=\cdots=\lambda_{1,n_1}=\mu_1$.

Finally, from ${\it trace}\,A_{JX_{1}}=0$ we get
$\lambda_{1,1}+(n-1)\mu_1=0$ as claimed.
\end{proof}

\noindent\textbf{Claim I-(3)}. If there exists a unit vector $V\in
T_pM_1^{n_1}$ such that $A_{JV}V=\lambda V$, then $\lambda$ has only a
finite number of possible values.

\begin{proof}[Proof of Claim I-(3)]

Assume that there exists a unit vector $V\in T_pM_1^{n_1}$ such that
$A_{JV}V=\lambda V$. Let $X_1=V$ and $\lambda_{1,1}=\lambda$, then
we may complete $X_1$ to obtain an orthonormal basis $\{X_i\}_{1\le
i\le n_1}$ of $T_pM_1^{n_1}$ such that, for each $2\le k\le n_1$,
$X_k$ is the eigenvector of $A_{JX_1}$ with eigenvalue
$\lambda_{1,k}$. Then we have \eqref{eqn:3.30}, from which we know the existence of an
integer $n_{1,1}$, $0\leq n_{1,1}\leq n_1-1$, such that, if
necessary after renumbering the basis, we have
\begin{equation}\label{eqn:3.48}
\left\{
\begin{aligned}
&\lambda_{1,2}=\cdots=\lambda_{1,n_{1,1}+1}=\frac{1}{2}
\Big(\lambda_{1,1}+\sqrt{\lambda_{1,1}^2+4\tilde{c}}\,\Big),\\
&\lambda_{1,n_{1,1}+2}=\lambda_{1,n_1}=\frac{1}{2}\Big(\lambda_{1,1}
-\sqrt{\lambda_{1,1}^2+4\tilde{c}}\,\Big).
\end{aligned}
\right.
\end{equation}

Similarly, we have \eqref{eqn:3.33}. By \eqref{eqn:3.48},
\eqref{eqn:3.33} and the fact that ${\it trace}\,A_{JX_1}=0$, we have
\begin{equation}\label{eqn:3.49}
\frac{1}{2}(n_1+n_2+1)\lambda_{1,1}+\frac{1}{2}(2n_{1,1}-n_1+1
+\varepsilon_1n_2)\sqrt{\lambda_{1,1}^2+4\tilde{c}}=0.
\end{equation}
This immediately implies that $\lambda_{1,1}$ has only finite
possibilities.
\end{proof}

\noindent\textbf{Claim I-(4)}. {\it The aforementioned tangent vector
$X_1$ at $(p,q)$ can be extended differentiably to a unit vector field,
still denoted by $X_1$, in a neighbourhood $U$ of $(p,q)$, such that for
each $(p',q')\in U$, $f_{(p',q')}$ defined on $U_{p'}M_1^{n_1}$ attains the
absolute maximum at $X_1(p',q')$.}

\begin{proof}[Proof of Claim I-(4)]

Let $\{E_1,\ldots,E_{n_1}\}$ be an arbitrary differentiable
orthonormal basis defined on a neighbourhood $U'$ of ${(p,q)}$ such that
$E_1{(p,q)}=X_1$. Then, from the fact $A_{JX_1}X_1=\lambda_{1,1}X_1$ at
${(p,q)}$, we define a function $\gamma$ by
$$
\begin{aligned}
\gamma:\ &\mathbb{R}^{n_1}\times U'\rightarrow
\mathbb{R}^{n_1},\\
&(a_1,\ldots,a_{n_1},(p',q'))\mapsto(b_1,\ldots,b_{n_1}),
\end{aligned}
$$
where
$b_k=b_k(a_1,\ldots,a_{n_1}):=\sum\limits_{i,j=1}^{n_1}a_ia_j\langle
A_{JE_i}E_j, E_k\rangle -\lambda_{1,1}a_k$ for $1\le k\le n_1$.

Using the fact that $f_{(p,q)}$ attains an absolute maximum in $E_1(p,q)$,
and that, by Claim I-(1), $A_{JE_1}E_k=\lambda_{1,k} E_k$ at $(p,q)$ for
$2\le k\le n_1$, we have the calculation that
\begin{equation*}
\begin{aligned}
\tfrac{\partial b_k}{\partial a_m}(1,0,\ldots,0,{(p,q)})
&=2\langle A_{JE_1{(p,q)}}E_m{(p,q)},E_k{(p,q)}\rangle -\lambda_{1,1}\delta_{km}\\
&=\left\{
\begin{aligned}
&0,\ \ {\rm if}\  k\neq m, \\
&\lambda_{1,1},\  \ {\rm if}\ k=m=1,\\
&2\lambda_{1,k}-\lambda_{1,1},\ \ {\rm if}\ k=m\ge2.
\end{aligned}
\right.
\end{aligned}
\end{equation*}

Given the fact that $\tilde{c}>0$, by \eqref{eqn:3.31} we have that
$2\lambda_{1,k}-\lambda_{1,1}\neq0$ for $k\ge2$. Hence the implicit
function theorem shows that there exist differentiable functions
$a_1,\ldots, a_{n_1}$, defined on a neighbourhood $U$ of $(p,q)$ and
satisfying
$$
a_1(p,q)=1,\ a_2(p,q)=0,\ \ldots,\ a_{n_2}(p,q)=0,
$$
such that
$$
\left\{
\begin{aligned}
&b_1(a_1(p',q'),\ldots,a_{n_1}(p',q'),
(p',q'))\equiv0,\\
&\ \ \ \ \cdots\\
&b_{n_1}(a_1(p',q'),\ldots,a_{n_1}(p',q'),(p',q'))\equiv0.
\end{aligned}
\right.
$$
Therefore, the local vector field $V$ defined by
$$
V=a_1E_1+\cdots+a_{n_1}E_{n_1}
$$
satisfies $V{(p,q)}=X_1$ and $A_{JV}V=\lambda_{1,1}V$.
Hence
\begin{equation}\label{eqn:3.50}
A_{J{\tfrac{V}{\sqrt{\langle V,V \rangle }}}}\tfrac{V}{\sqrt{\langle
V,V \rangle }} =\tfrac{\lambda_{1,1}}{\sqrt{\langle V,V \rangle
}}\tfrac{V}{\sqrt{\langle V,V \rangle }}.
\end{equation}

According to Claim I-(3), there is a finite number of possible values that the function
$\tfrac{\lambda_{1,1}}{\sqrt{\langle V,V\rangle}}$ can take. On the other hand, since
$\tfrac{\lambda_{1,1}}{\sqrt{\langle V,V \rangle}}$ is continuous
and $\langle V,V \rangle (p,q)=1$, it must be that $\langle
V,V\rangle=1$ identically. Define on $U$ a vector field $X_1:=V$. By
Claim I-(1) and its proof we know that, for any point $(p',q')\in U$,
$f_{(p',q')}$ attains an absolute maximum at $X_1(p',q')$. This verifies the
assertion of Claim I-(4).
\end{proof}
Finally, having determined the unit vector field $X_1$ as in Claim
I-(4), we further choose vector fields $X_2,\ldots,X_{n_1}$ (which
are orthogonal to $X_1$) such that $\{X_i\}_{1\le i\le n_1}$ is a
local orthonormal frame of $TM_1^{n_1}$. Then, combining with Lemma
\ref{lem:3.2}, we  complete immediately the proof for the first
step of induction.


\vskip 3mm

\noindent{\bf The second step of induction.}

\vskip 1mm
In this step, we first assume the assertion of Lemma \ref{lem:3.5}
for all $i\le k$, where $k\in \{2,\ldots,n_1-1\}$ is a fixed
integer. Therefore, there exists a local orthonormal frame of vector
fields $\{X_i\}_{1\le i\le n_1}$ of $M_1^{n_1}$, such that the
operator $A_J$ takes the following form:
\begin{equation}\label{eqn:3.51}
\left\{
\begin{aligned}
&A_{JX_1}X_1=\lambda_{1,1} X_1,\\
&A_{JX_i}X_i=\mu_1 X_1+\cdots+\mu_{i-1} X_{i-1}+\lambda_{i,i}X_i,\ 1< i\leq k,\\
&A_{JX_i}X_j=\mu_i X_j,\ 1\leq i\leq k, \ i<j\leq n_1,    \\
&A_{JX_i}Y=\mu_iY,\ 1\leq i\leq k,\ Y\in TM_2^{n_2},
\end{aligned}
\right.
\end{equation}
where $\mu_i$ and $\lambda_{i,i}$ for $1\le i\le k$ are constants
that satisfy the relations:
\begin{equation}\label{eqn:3.52}
\lambda_{i,i}+(n-i)\mu_i=0,\ \ 1\le i\le k.
\end{equation}
Moreover, for $1\leq i\leq k$ and $(p',q')$ around $(p,q)$, $\lambda_{i,i}$ is
the maximum of $f_{(p',q')}$ defined on
$$
\{u\in T_{p'}M_1^{n_1} \mid \langle u,u\rangle=1, u\perp
X_1,\ldots,X_{i-1}\}.
$$
Then, as purpose of the second step, we should verify the assertion
of Lemma \ref{lem:3.5} for $i=k+1$. To do so, we have to show that
there exists a local orthonormal frame of vector fields
$\{\tilde{X}_i\}_{1\leq i\leq n_1}$ of $TM_1^{n_1}$ given by
$$
\tilde{X}_1=X_1,\ldots,\tilde{X}_k=X_k;\ \
\tilde{X}_l=\sum_{t=k+1}^{n_1}T^t_lX_t,\ k+1\le l\le n_1,
$$
such that $T=(T_l^t)_{k+1\le l,t\le n_1}$ is an orthogonal matrix,
and the operator $A_J$ takes the following form:
\begin{equation}\label{eqn:3.53}
\left\{
\begin{aligned}
&A_{J\tilde{X}_1}\tilde{X}_1=\lambda_{1,1} \tilde{X}_1,\\
&A_{J\tilde{X}_i}\tilde{X}_i=\mu_1 \tilde{X}_1+\cdots+\mu_{i-1} \tilde{X}_{i-1}+\lambda_{i,i}\tilde{X}_i,\ 2\le i\leq k+1,\\
&A_{J\tilde{X}_i}\tilde{X}_j=\mu_i \tilde{X}_j,\ 1\leq i\leq k+1, \ i+1\le j\leq n_1,    \\
&A_{J\tilde{X}_i}Y=\mu_iY,\ 1\leq i\leq k+1,\ Y\in TM_2^{n_2},
\end{aligned}
\right.
\end{equation}
where $\mu_i$ and $\lambda_{i,i}$ for $1\le i\le k+1$ are constants
and satisfy the relations
\begin{equation}\label{eqn:3.54}
\lambda_{i,i}+(n-i)\mu_i=0,\ \ 1\le i\le k+1.
\end{equation}
Moreover, for $1\leq i\leq k+1$ and $(p',q')$ around $(p,q)$, $\lambda_{i,i}$ is the maximum of
$f_{(p',q')}$ defined on
$$
\{u\in T_{p'}M_1^{n_1} \mid \langle u,u\rangle=1, u\perp
\tilde{X}_1,\ldots,u\perp\tilde{X}_{i}\}.
$$

\vskip 2mm
Similarly to the first step, the proof of the above conclusion
will also be divided into the verification of four claims.

\vskip 2mm

\noindent\textbf{Claim II-(1)}. {\it For any $(p,q)\in
M_1^{n_1}\times M_2^{n_2}$, there exists an orthonormal basis
$\{\bar{X}_i\}_{1\leq i\leq n_1}$ of $T_pM_1^{n_1}$ and real numbers
$\lambda_{k+1,k+1}>0$, $\lambda_{k+1,k+2}=\cdots=\lambda_{k+1,n_1}$
and $\mu_{k+1}$, such that the following relations hold:}
\begin{equation*}
\left\{
\begin{aligned}
&A_{J\bar{X}_1}\bar{X}_1=\lambda_{1,1} \bar{X}_1,\\
&A_{J\bar{X}_i}\bar{X}_i=\mu_1 \bar{X}_1+\cdots+\mu_{i-1} \bar{X}_{i-1}+\lambda_{i,i}X_i,\ 2\le i\leq k+1,\\
&A_{J\bar{X}_{k+1}}\bar{X}_i=\lambda_{k+1,i}\bar{X}_i,\,\  i\ge k+2,\\
&A_{J\bar{X}_{k+1}}Y=\mu_{k+1}Y,\ \ Y\in T_qM_2^{n_2}.
\end{aligned}
\right.
\end{equation*}

\begin{proof}[Proof of Claim II-(1)]

By the induction assumption, we have an orthonormal basis
$\{X_i\}_{1\le i\le n_1}$  such that \eqref{eqn:3.51} and \eqref{eqn:3.52} hold.
We first take $\bar{X}_1=X_1(p,q),\ldots,\bar{X}_k=X_k(p,q)$. Then,
putting
$$
V_k=\{u\in T_{p}M_1^{n_1} \mid u\perp \bar{X}_1,\dots,u\perp
\bar{X}_k\},
$$
we will show that, restricting on $U_pM_1^{n_1}\cap V_k$, the
function $f_{(p,q)}\not=0$.

Indeed, suppose on the contrary that $f_{(p,q)}\,|_{V_k}=0$. Then, letting
$\{u_i\}_{k+1\le i\le n_1}$ be an orthonormal basis of $V_k$, we
have $\langle A_{Ju_i}u_{j},u_k\rangle=0$, $k+1\leq i,j,k\leq n_1$.
Taking in \eqref{eqn:2.4} that $X=u_{k+2},Y=Z=u_{k+1}$, by assumption
of induction and Lemma \ref{lem:3.2}, we obtain
$\mu_1^2+\cdots+\mu_{k}^2+\tilde{c}=0$. This is a contradiction to
the fact $\tilde c>0$.

Now, we can choose $\bar{X}_{k+1}$ such that $f_{(p,q)}$, restricted on
$U_pM_1^{n_1}\cap V_k$, attains its maximum with value
$$
\lambda_{k+1,k+1}:=\langle
A_{J\bar{X}_{k+1}}\bar{X}_{k+1},\bar{X}_{k+1}\rangle>0.
$$
Consider the self-adjoint operator $\mathcal{A}:\ V_k\to V_k$
defined by
\begin{equation*}
\mathcal{A}(X)=A_{J\bar{X}_{k+1}}X- \sum_{i=1}^k\langle
A_{J\bar{X}_{k+1}}X,\bar{X}_i \rangle \bar{X}_i.
\end{equation*}
It is easy to see that
$\mathcal{A}(\bar{X}_{k+1})=\lambda_{k+1,k+1}\bar{X}_{k+1}$. Hence,
by the assumption of induction, we have:
\begin{align*}
\lambda_{k+1,k+1}\bar{X}_{k+1}=&A_{J\bar{X}_{k+1}}\bar{X}_{k+1}-\sum_{i=1}^k \langle A_{J\bar{X}_{k+1}}\bar{X}_{k+1},\bar{X}_i\rangle \bar{X}_i\\
=& A_{J\bar{X}_{k+1}}\bar{X}_{k+1}-\sum_{i=1}^k\langle A_{J\bar{X}_{i}}\bar{X}_{k+1},\bar{X}_{k+1}\rangle \bar{X}_i\\
=& A_{J\bar{X}_{k+1}}\bar{X}_{k+1}-\sum_{i=1}^k\mu_i \bar{X}_i.
\end{align*}

Next, we choose $\bar{X}_{k+2}, \ldots,\bar{X}_{n_1}$ as the
remaining unit eigenvectors of $\mathcal{A}$, with corresponding
eigenvalues $\lambda_{k+1,k+2}$, $\ldots$, $\lambda_{k+1,n_1}$,
respectively. Thus, by Lemma \ref{lem:3.2} we have $\mu_{k+1}$, and
the following relations:
\begin{equation}\label{eqn:3.55}
\left\{
\begin{aligned}
&A_{J\bar{X}_{k+1}}\bar{X}_{k+1}=\mu_1 \bar{X}_1+\cdots+\mu_{k} \bar{X}_{k}+\lambda_{k+1,k+1}\bar{X}_{k+1},\\
&A_{J\bar{X}_{k+1}}\bar{X}_i=\lambda_{k+1,i}\bar{X}_i,\,\ k+2\le i\leq n_1,\\
&A_{J\bar{X}_{k+1}}Y=\mu_{k+1}Y,\ \ Y\in T_qM_2^{n_2}.
\end{aligned}
\right.
\end{equation}

Now, taking in \eqref{eqn:2.4} that $X=Z=\bar{X}_{k+1}$ and
$Y=\bar{X}_j$ with $j\ge k+2$, combining with \eqref{eqn:2.12}, we can
obtain
\begin{equation}\label{eqn:3.56}
\lambda_{k+1,j}^2-\lambda_{k+1,
k+1}\lambda_{k+1,j}-\tilde{c}-(\mu_1^2+\cdots+\mu_{k}^2)=0.
\end{equation}
It follows that
\begin{equation}\label{eqn:3.57}
\begin{aligned}
&\lambda_{k+1,k+2}=\cdots=\lambda_{k+1,n_1}\\
&=\frac{1}{2}\Big(\lambda_{k+1,k+1}-\sqrt{\lambda_{k+1,k+1}^2
        +4(\tilde{c}+\mu_1^2+\cdots+\mu_{k-1}^2+\mu_k^2)}\,\Big).
\end{aligned}
\end{equation}

On the other hand, taking in \eqref{eqn:2.4} that $X=Z=\bar{X}_{k+1}$,
and $Y\in T_qM_2^{n_2}$ be a unit vector, combining with
\eqref{eqn:2.12}, we can obtain
\begin{equation}\label{eqn:3.58}
\mu_{k+1}^2-\lambda_{k+1,k+1}\mu_{k+1}-\tilde{c}-(\mu_1^2+\cdots+\mu_{k}^2)=0.
\end{equation}

Hence
\begin{equation}\label{eqn:3.59}
\mu_{k+1}=\frac{1}{2}\Big(\lambda_{k+1,k+1}+\varepsilon_{k+1}
\sqrt{\lambda_{k+1,k+1}^2+4(\tilde{c}+\mu_1^2+\cdots+\mu_{k}^2)}\Big),
\end{equation}
where $\varepsilon_{k+1}=\pm1$.
Then, using that  ${\it trace}\,A_{J\bar{X}_{k+1}}=0$, we get
$n_1-n_2\varepsilon_{k+1}-k-1>0$ and
\begin{align}\label{eqn:3.60}
&\lambda_{k+1,k+1}=2\sqrt{\tfrac{\tilde{c}+\mu_1^2+\cdots
+\mu_{k-1}^2+\mu_k^2}{\big(\tfrac{n_1+n_2-k+1}{n_1-n_2\varepsilon_{k+1}-k-1}\big)^2-1}}.
\end{align}
By the assumption that $\mu_1,\ldots,\mu_k$ are constants we see
that, as claimed, $\lambda_{k+1,k+2}= \cdots=\lambda_{k+1,n_1}$ and
$\mu_{k+1}$ are also constants.
\end{proof}


\noindent\textbf{Claim II-(2)}. {\it
$\lambda_{k+1,k+2}=\cdots=\lambda_{k+1,n_1}=\mu_{k+1}$ and
$\lambda_{k+1,k+1}+(n-k-1)\mu_{k+1}=0$.}

\begin{proof}[Proof of Claim II-(2)]

From \eqref{eqn:3.57} and \eqref{eqn:3.59}, the first assertion is
equivalent to showing that $\varepsilon_{k+1}=-1$. Suppose, on the
contrary, that $\varepsilon_{k+1}=1$. Then we have
\begin{equation}\label{eqn:3.61}
\mu_{k+1}\lambda_{k+1,i}=-(\tilde{c}+\mu_1^2+\cdots+\mu_{k}^2),\ \
i\ge k+2.
\end{equation}

Similar to getting \eqref{eqn:3.60}, now we have
\begin{equation}\label{eqn:3.62}
n_1-n_2-k-1>0
\end{equation}
and
\begin{equation}\label{eqn:3.63}
\lambda_{k+1,k+1}=2\sqrt{\tfrac{\tilde{c}+\mu_1^2+\cdots
+\mu_{k}^2}{\big(\tfrac{n_1+n_2-k+1}{n_1-n_2-k-1}\big)^2-1}}.
\end{equation}
Put
$$
V_{k+1}=\{u\in T_{p}M_1^{n_1} \mid u\perp \bar{X}_1,\dots,u\perp
\bar{X}_{k+1}\}.
$$
Then, a similar argument as in the proof of Claim II-(1) shows that,
restricting on $U_pM_1^{n_1}\cap V_{k+1}$, the function $f_{(p,q)}\not=0$.\\
Now, by a totally similar process as in the proof of Claim II-(1),
we can choose another orthonormal basis $\{X'_i\}_{1\le i\le n_1}$
of $T_pM_1^{n_1}$ with $X'_j=\bar{X}_j$ for $1\le j\le k+1$, such
that $f_{(p,q)}$, restricting on $U_pM_1^{n_1}\cap V_{k+1}$, attains its
maximum $\lambda_{k+2,k+2}>0$ at $X'_{k+2}$ so that
$\lambda_{k+2,k+2} =h(A_{JX'_{k+2}}X'_{k+2},X'_{k+2})$.\\
As before, we define a self-adjoint operator $\mathcal{A}:\
V_{k+1}\to V_{k+1}$ by
\begin{equation*}
\mathcal{A}(X)=A_{JX'_{k+2}}X- \sum_{i=1}^{k+1}\langle
A_{JX'_{k+2}}X,X'_i \rangle X'_i.
\end{equation*}
Then we have $\mathcal{A}(X'_{k+2})=\lambda_{k+2,k+2}X'_{k+2}$. As
before we will choose $X'_{k+3}, \ldots,X'_{n_1}$ as the remaining
unit eigenvectors of $\mathcal{A}$, with corresponding eigenvalues
$\lambda_{k+2,k+3}$, $\ldots$, $\lambda_{k+2,n_1}$, respectively.
In this way, we can prove that
\begin{equation}\label{eqn:3.64}
\left\{
\begin{aligned}
&A_{JX'_{k+2}}X'_{k+2}=\mu_1 X'_1+\cdots+\mu_{k}X'_{k}+\lambda_{k+1,k+2}X'_{k+1}+\lambda_{k+2,k+2}X'_{k+2},\\
& A_{JX'_{k+2}}X'_i=\lambda_{k+2,i}X'_i,\ \ k+3\leq i\leq n_1.
\end{aligned}
\right.
\end{equation}

Taking $X=Z=X'_{k+2}$ and $Y=X'_i$ for $k+3\le i\leq n_1$ in \eqref{eqn:2.4}
and using \eqref{eqn:2.12}, we  obtain
\begin{equation}\label{eqn:3.65}
\lambda_{k+2,i}^2-\lambda_{k+2,k+2}\lambda_{k+2,i}
-\tilde{c}-(\mu_1^2+\cdots+\mu_{k}^2+\lambda_{k+1,i}^2)=0,\ \ k+3\leq i\leq n_1.
\end{equation}
Noting that for $k+3\leq i\leq n_1$ we have $\lambda_{k+2,k+2}\ge
2\lambda_{k+2,i}$, it follows from \eqref{eqn:3.65} that
\begin{equation}\label{eqn:3.66}
\begin{aligned}
\lambda_{k+2,i}=\frac{1}{2}\Big(\lambda_{k+2,k+2}-&\sqrt{\lambda_{k+2,k+2}^2
+4(\tilde{c}+\mu_1^2+\cdots+\mu_{k}^2+\lambda_{k+1,i}^2)}\,\Big),\\
&\hspace{5cm}i\ge k+3.
\end{aligned}
\end{equation}

Similarly, let $X=Z=X'_{k+2}$ and $Y\in T_qM_2^{n_2}$ be a unit
vector in \eqref{eqn:2.4}. Using \eqref{eqn:2.12} we get
\begin{equation}\label{eqn:3.67}
\mu_{k+2}^2-{\mu_{k+2}}\lambda_{k+2,k+2}-\tilde{c}
-(\mu_1^2+\cdots+\mu_{k}^2+\lambda_{k+1,i}\mu_{k+1})=0,\ i\ge k+2.
\end{equation}
Combining \eqref{eqn:3.61} and \eqref{eqn:3.67} we obtain
\begin{equation}\label{eqn:3.68}
{\mu^2_{k+2}}-{\mu_{k+2}}\lambda_{k+2,k+2}=0,
\end{equation}
and therefore it holds that
\begin{equation}\label{eqn:3.69}
\mu_{k+2}=\frac{1}{2}(\lambda_{k+2,k+2}+\varepsilon_{k+2}\lambda_{k+2,k+2}),\ \varepsilon_{k+2}=\pm1.
\end{equation}
Then, using ${\it trace}\,A_{JX'_{k+2}}=0$, we can get $n_1-k-2>0$ and
\begin{equation}\label{eqn:3.70}
\lambda_{k+2,k+2}=2\sqrt{\tfrac{\tilde{c}+\mu_1^2+\cdots
+\mu_{k}^2+\lambda_{k+1,i}^2}{\big(\tfrac{n_1+n_2-k+\varepsilon_{k+2}n_2}{n_1-k-2}\big)^2-1}}\,,\
\ i\ge k+2.
\end{equation}
Given \eqref{eqn:3.62}, we have the following calculations
\begin{equation}\label{eqn:3.71}
\begin{aligned}
\tfrac{n_1+n_2-k+1}{n_1-n_2-k-1}-\tfrac{n_1+n_2+\varepsilon_{k+2}n_2-k}{n_1-k-2}&>\tfrac{n_1+n_2-k+1}{n_1-n_2-k-1}-\tfrac{n_1+2n_2-k}{n_1-k-2}\\
&=\tfrac{2(n_2+1)(n_2-1)}{(n_1-n_2-k-1)(n_1-k-2)}\,.
\end{aligned}
\end{equation}
Then, by \eqref{eqn:3.63} and \eqref{eqn:3.70}, we get
$\lambda_{k+2,k+2}>\lambda_{k+1,k+1}$, which is a contradiction.
Therefore, $\varepsilon_{k+1}=-1$ and
$\lambda_{k+1,k+2}=\cdots=\lambda_{k+1,n_1}=\mu_{k+1}$, as claimed.

Finally, from ${\it trace}\ A_{J\bar{X}_{k+1}}=0$, we get
$$
\lambda_{k+1,k+1}+(n-k-1)\mu_{k+1}=0.
$$
This completes the verification of {\bf Claim II-(2)}.
\end{proof}


\vskip 2mm

\noindent\textbf{Claim II-(3)}. {\it Let $\{X_i\}_{1\le i\le n_1}$
be the local orthonormal vector fields of $M^n$ which form a basis for
the first component as in the assumption of induction. If a unit vector
field $V$ of $TM_1^{n_1}\setminus {\rm span}\{X_1,\ldots,X_k\}$ has
the property that $A_{JV}V=\lambda V+\mu_1X_1+\cdots+\mu_kX_k$, then
the function $\lambda$ takes values of only finite possibilities.}
\begin{proof}[Proof of Claim II-(3)]

We first carry the discussion at an arbitrary fixed point $(p,q)$ Let
$X'_{k+1}:=V$, $X'_1=X_1,\ldots,X'_k=X_k$, $\lambda_{k+1,k+1}:=\lambda$.

Put $V_k=\{u\in T_{p}M_1^{n_1} \mid u\perp X_1,\ldots,u\perp
X_k\}$. Define $\mathcal{A}: V_k\to V_k$ by
$$
\mathcal{A}(X)=A_{JV}X-\sum_{i=1}^k\langle A_{JV}X,
X_i\rangle X_i.
$$
It is easily seen that $\mathcal{A}$ is a self-adjoint transformation and that
$\mathcal{A}(V)=\lambda V$. Thus, we can choose an orthonormal basis
$\{X'_i\}_{k+1\le i\le n_1}$ of $V_k$, such that $\mathcal{A}(X'_i)=\lambda_{i,i}X'_i$
for $k+2\le i\le n_1$. Then, as before we see that \eqref{eqn:3.56} holds, and
thus there exists an integer $n_{1,k+1}$, $0\leq n_{1,k+1}\leq n_1-(k+1)$
such that, if necessary after renumbering the basis, we have
\begin{equation}\label{eqn:3.72}
\left\{
\begin{aligned}
&\lambda_{k+1,k+2}=\cdots=\lambda_{k+1,n_{1,k+1}+k+1}\\
&\ =\frac{1}{2}\Big(\lambda_{k+1,k+1}+\sqrt{\lambda_{k+1,k+1}^2
        +4(\tilde{c}+\mu_1^2+\cdots+\mu_{k-1}^2+\mu_k^2)}\Big),\\
&\lambda_{k+1, n_{1,k+1}+k+2}=\cdots=\lambda_{k+1,n_1}\\
&\ =\frac{1}{2}\Big(\lambda_{k+1,k+1}-\sqrt{\lambda_{k+1,k+1}^2
        +4(\tilde{c}+\mu_1^2+\cdots+\mu_{k-1}^2+\mu_k^2)}\Big).
\end{aligned}\right.
\end{equation}
Then, using ${\it trace}\,A_{JX'_{k+1}}=0$, we can show that
\begin{equation}\label{eqn:3.73}
\begin{aligned}
&\lambda_{k+1,k+1}=2\sqrt{\tfrac{\tilde{c}+\mu_1^2+\cdots
+\mu_{k-1}^2+\mu_k^2}{\big(\tfrac{n_1+n_2-k+1}{2n_{1,k+1}-n_1+n_2\varepsilon_{k+1}+k+1}\big)^2-1}}.
\end{aligned}
\end{equation}

Finally, noticing that by assumption $\mu_1,\ldots,\mu_k$ are
constants, and that the set
$$
\big\{n_{1,k+1}(p)\,|\,p\in M_1^{n_1}\big\}
$$
consists of finite numbers, we get the assertion that $\lambda=\lambda_{k+1,k+1}$
takes values of only finite possibilities.
\end{proof}

\noindent\textbf{Claim II-(4)}. {\it Let $\{X_i\}_{1\le i\le n_1}$
be the local vector fields on $M^n$ as in the assumption
of induction, $V_k=\{u\in T_pM_1^{n_1} \mid \langle u, u\rangle=1,
u\perp X_1, \dots, u\perp X_k\}$. The unit vector $\bar{X}_{k+1}\in
T_pM_1^{n_1}$ determined in Claim II-(1) can be extended differentiably
to be a unit vector field, denoted by $\tilde{X}_{k+1}$, in a neighbourhood
$U$ of $(p,q)$, such that for each $(p', q')\in U$, $f_{(p', q')}$ defined
on $V_k$ attains the absolute maximum at $\tilde{X}_{k+1}{(p', q')}$.}

\begin{proof}[Proof of Claim II-(4)]

Let $\{E_{k+1},\ldots,E_{n_1}\}$ be
arbitrary differentiable orthonormal vector fields of $V_k$ defined on a
neighbourhood $U'$ of $(p,q)$ such that $E_{k+1}(p,q)=\bar{X}_{k+1}$.
Then, we define a function $\gamma$ by
$$
\begin{aligned}
\gamma:\ &\mathbb{R}^{n_1-k}\times U'\rightarrow
\mathbb{R}^{n_1-k},\\
&(a_{k+1}, \ldots,a_{n_1},(p',q'))\mapsto(b_{k+1},\ldots,b_{n_1}),
\end{aligned}
$$
where $b_l=\sum_{i,j=k+1}^{n_1}a_ia_j\langle
A_{JE_i}E_j,E_l\rangle-\lambda_{k+1,k+1}a_l,\ l=k+1\le l\le n_1$.
Using the fact that $f_{(p,q)}$ attains an absolute maximum in
$E_{k+1}{(p,q)}$ so that
$$
\langle A_{JE_{k+1}}E_{l},E_{l}\rangle|_{(p,q)}=\lambda_{k+1,l},\ \ l\ge
k+1,
$$
we then obtain that
\begin{equation*}
\begin{aligned}
\tfrac{\partial b_l}{\partial a_m}(1,0,\ldots,0,(p,q))&=2 \langle
A_{JE_{k+1}{(p,q)}}
E_m{(p,q)},E_l{(p,q)}-\lambda_{k+1,k+1}\delta_{lm}\\
&=\left\{
\begin{aligned}
&0,\ \ {\rm if}\  l\neq m ,\\
&\lambda_{k+1,k+1},\ \ {\rm if}\ l= m=k+1 ,\\
&2\lambda_{k+1,l}-\lambda_{k+1,k+1},\ \ {\rm if}\ l= m\ge k+2.
\end{aligned}
\right.
\end{aligned}
\end{equation*}
As $\tilde{c}>0$, then from \eqref{eqn:3.57} we obtain that
$2\lambda_{k+1,l}-\lambda_{k+1,k+1}\neq0$.
Hence, similar to the proof of Claim I-(4), the implicit function
theorem shows that there exist differentiable functions
$a_{k+1},\ldots, a_{n_1}$, defined on a neighbourhood $U$ of $(p,q)$,
such that the local vector field $V$, defined by
$$
V=a_{k+1}E_{k+1}+\cdots+a_{n_1}E_{n_1},
$$
has the property $V(p,q)=X_{k+1}$ and satisfies
that
$$
A_{JV}V=\lambda_{k+1, k+1}V+\mu_1\langle V,V\rangle X_1+\cdots+\mu_k
\langle V,V\rangle X_k.
$$
Hence
\begin{equation}\label{eqn:3.74}
A_{J\tfrac{V}{\sqrt{\langle V,V \rangle }}}
\tfrac{V}{\sqrt{\langle V,V \rangle }}
=\tfrac{\lambda_{k+1,k+1}}{\sqrt{\langle V,V \rangle}}
\tfrac{V}{\sqrt{\langle V,V \rangle }}+\mu_1X_1+\cdots+\mu_kX_k.
\end{equation}

According to Claim II-(3), the function
$\tfrac{\lambda_{k+1,k+1}}{\sqrt{\langle V,V \rangle }}$ can
take a finite number of values. On the other hand,
$\tfrac{\lambda_{k+1,k+1}}{\sqrt{\langle V, V \rangle}}$ is
continuous and $\langle V,V \rangle {(p,q)}=1$. Thus $\langle V,V
\rangle =1$ holds identically. Let $\tilde{X}_{k+1}:=V$. Then,
\eqref{eqn:3.74} and $\langle V,V \rangle=1$ imply that for any
$(p',q')\in U$, $f_{(p',q')}$ defined on $V_k(p',q')$ attains an absolute maximum
at $\tilde{X}_{k+1}(p',q')$.\end{proof}

Finally, we choose vector fields
$\tilde{X}_1=X_1,\ldots,\tilde{X}_k=X_k$ and
$\tilde{X}_{k+2},\ldots,\tilde{X}_{n_1}$ such that
$\{\tilde{X}_1,\tilde{X}_2,\ldots,\tilde{X}_{n_1}\}$ are
orthonormal vector fields of $M^n$ which together span a basis for
the first component of the tangent space. Then, combining with Lemma
\ref{lem:3.2}, we immediately fulfil the second step of induction.

Accordingly, we have completed the proof of Lemma \ref{lem:3.5}.
\end{proof}

In the following part, we aim at giving the explicit parametrization
of $\psi:\,M^n\to \tilde M^n(4\tilde{c})$. For this we will use
Theorems \ref{thm:2.1} and \ref{thm:2.2} from \cite{LW2}.

Firstly, we will prove that the submanifold $M^n$ has parallel
second fundamental form. We will do this by direct computations: for
the local orthonormal frame $\{X_i\}_{1\le i\le n_1}$ of $M_1^{n_1}$
as determined in Lemma \ref{lem:3.5}, we will use the Codazzi
equation in \eqref{eqn:2.5} to show that, for each $1\le i\le n_1$,
$X_i$ is a parallel vector field. Then we will further prove that
$\psi:\,M^n\to \tilde M^n(4\tilde{c})$ has parallel second fundamental form.

\begin{lem}\label{lem:3.6}
Let $\{X_1,\ldots,X_{n_1}\}$ be the local orthonormal vector fields of
$M^n$, as determined in Lemma \ref{lem:3.5} and let $\{Y_1,\ldots,Y_{n_2}\}$
be a local vector fields on $M$ which form a basis for the second component,
and moreover, we assume that $\{Y_i\}$ depend only on the
second component. Then
\begin{equation*}
\nabla X_i=0,\ \ 1\le i\le n_1.
\end{equation*}
\end{lem}
\begin{proof}
We will proceed by induction on the subscript of $X_i$ and prove separately
that $ \nabla_X X_i=0,\ X\in TM_1^{n_1}$ and $\nabla_Y X_i=0,\ Y\in TM_2^{n_2}$,
where $ 1\le i\le n_1.$

Let us check first that $\nabla_X X_i=0,\ X\in TM_1^{n_1}$.

For $i\ge2$, by using \eqref{eqn:2.3} and \eqref{eqn:3.28},
we have
$$
\left\{
\begin{aligned}
&J(\nabla h)(X_i,X_1,X_1)=(2\mu_1-\lambda_{1,1})\nabla_{X_i} X_1,\\
&J(\nabla h) (X_1,X_i,X_1)=-\mu_1\nabla_{X_1}
X_i+A_{JX_1}(\nabla_{X_1}X_i)+A_{JX_i}(\nabla_{X_1}X_1).
\end{aligned}
\right.
$$
Then, the Codazzi equations $J(\nabla h)(X_i,X_1,X_1)=J(\nabla h)
(X_1,X_i,X_1)$ give that
\begin{equation}\label{eqn:3.75}
(2\mu_1-\lambda_{1,1})\nabla_{X_i} X_1=-\mu_1\nabla_{X_1}
X_i+A_{JX_1}(\nabla_{X_1}X_i)+A_{JX_i}(\nabla_{X_1}X_1).
\end{equation}

Taking the component in the direction of $X_1$ in \eqref{eqn:3.75} we can
get $\nabla_{X_1} X_1=0$.
Substituting $\nabla_{X_1} X_1=0$ into \eqref{eqn:3.75}, and then taking
the component in the direction of $X_i$, we get $\langle
\nabla_{X_i} X_1,X_k\rangle=0$ for $2\leq i,k\leq n_1$.

The above facts immediately verify for the first step of induction that
$$
\nabla_X X_1=0,\ X\in TM_1^{n_1}.
$$

\vskip 2mm

Next, assume by induction that for a fixed $j\ge2$ it holds
\begin{equation}\label{eqn:3.76}
\nabla_X X_k=0,\ X\in TM_1^{n_1},\ k=1,\ldots, j-1.
\end{equation}

We claim that $\nabla X_j=0$. The proof of the claim will be given
in four cases:

(1) From the induction assumption and the fact that $\langle X_i,X_l
\rangle=\delta_{il}$, we get
$$
\langle \nabla_{X_i}X_j,X_k\rangle=-\langle
\nabla_{X_i}X_k,X_j\rangle=0,\ 1\leq i\leq n_1,\ k\leq j.
$$

\vskip 2mm

(2) For $i\le j-1$, by the induction assumption we have
$$
\begin{aligned}
J(\nabla h) (X_i,X_j,X_j)& =-\nabla_{X_i}A_{JX_j}X_j+2A_{JX_j}\nabla_{X_i}X_j\\
&=\lambda_{j,j}\nabla_{X_i}X_j-2A_{JX_j}\nabla_{X_i}X_j\\
&=(\lambda_{j,j}-2\mu_j)\nabla_{X_i}X_j;\\
J(\nabla h) (X_j,X_i,X_j)& =-\nabla_{X_j}A_{JX_i}
X_j+A_{JX_j}\nabla_{X_j}X_i+A_{JX_i}\nabla_{X_j}X_j\\
&=-\mu_i\nabla_{X_j}X_j+A_{JX_j}\nabla_{X_j}X_i+A_{JX_i}\nabla_{X_j}X_j\\
&=-\mu_i \nabla_{X_j}X_j +A_{JX_i}\nabla_{X_j}X_j.
\end{aligned}
$$

Then, by $J(\nabla h) (X_i,X_j,X_j)=J(\nabla h) (X_j,X_i,X_j)$, we
immediately get
\begin{equation*}
\langle \nabla_{X_i}X_j, X_{j_0}\rangle=0, \ \ i\le j-1,\ j+1\leq
j_0\leq n_1.
\end{equation*}

\vskip 2mm

(3) For $j+1\leq j_0\leq n_1$, similar and direct calculations give
that
$$
\begin{aligned}
J(\nabla h) (X_{j_0},X_j,X_j)& =-\nabla_{X_{j_0}}A_{JX_j}X_j+2A_{JX_j}\nabla_{X_{j_0}}X_j\\
&=\lambda_{j,j}\nabla_{X_{j_0}}X_j-2A_{JX_j}\nabla_{X_{j_0}}X_j\\
&=(\lambda_{j,j}-2\mu_j)\nabla_{X_{j_0}}X_j;\\
J(\nabla h) (X_j,X_{j_0},X_j)& =-\nabla_{X_j}A_{JX_{j_0}}X_j+A_{JX_j}\nabla_{X_j}X_{j_0}+A_{JX_{j_0}}\nabla_{X_j}X_j\\
&=-\mu_j \nabla_{X_j}X_{j_0}+A_{JX_j}\nabla_{X_j}X_{j_0}+A_{JX_{j_0}}\nabla_{X_j}X_j.
\end{aligned}
$$

By $J(\nabla h) (X_j,X_{j_0},X_j)=J(\nabla h) (X_j,X_{j_0},X_j)$ and
taking the component in the direction of $X_{j}$, we obtain that
\begin{equation*}
\langle \nabla_{X_j}X_j, X_{j_0}\rangle=0,\ \ j+1\leq j_0\leq n_1.
\end{equation*}

\vskip 2mm

(4) For $i\ge j+1$, by similar calculations for both sides of
$$
J(\nabla h)(X_i,X_{j},X_j)=J(\nabla h)(X_j, X_{i},X_j),
$$
and taking the component in the direction of $X_{j_0}$ for $j_0\ge
j+1$, we can get
$$
\langle \nabla_{X_{i}} X_j,X_{j_0} \rangle=0,\ \ i\ge j+1,\ j_0\ge
j+1.
$$
Summing up the above four cases, we finally get the assertion
$$
\nabla_X X_j=0,\ X\in TM_1^{n_1}.
$$
Finally, we must prove that $\nabla_Y X_i=0$ for $Y\in TM_2^{n_2}$ and
$1\le i\le n_1$. The proof follows the same steps as before. For instance,
we start with the Codazzi equation $J(\nabla h)(X_i, Y_1,X_1)=J(\nabla h)(Y_1,X_i,X_1)$,
$i>1$. Multiplying once by $X_1$ and once by $Y_j$, $j\le n_1$, we get
that $\nabla_{Y_1}X_1=0$. Then, $\nabla_{Y_j}X_1=0$, $j>1$ follows similarly
from  $J(\nabla h)(Y_i, X_1,X_1)=J(\nabla h)(X_1,Y_i,X_1)$, $i>1$. We then
complete the proof of this part by following the same steps as for $\nabla_X X_i=0,\ X\in TM_1^{n_1}$.

By induction we have completed the proof of Lemma \ref{lem:3.6}.
\end{proof}

\begin{lem}\label{lem:3.7}
Under the condition of Theorem \ref{thm:3.1}, the submanifold
$\psi: M^n\to \tilde M^n(4\tilde{c})$ has parallel second
fundamental form: $\nabla h=0$.
\end{lem}

\begin{proof}
We have that $M^n=M_1^{n_1}(c_1)\times M_2^{n_2}(c_2)$, for $c_1=0$, $c_2>0$ and
$\tilde{c}=1$. Let $\{X_i\}_{1\le i\le n_1}$ and $\{Y_j\}_{1\le j\le n_2}$ be the
local orthonormal frames of vector fields of $M_1^{n_1}$ and $M_2^{n_2}$,
respectively, as described in Lemma \ref{lem:3.5}.
Consider arbitrarily $X\in TM_1^{n_1}$ and $Y\in TM_2^{n_2}$. We will make
use of the Codazzi equation \eqref{eqn:2.5}, equations \eqref{eqn:3.17}, \eqref{eqn:3.28}
and the fact that $\nabla X_i=0,1\leq i\leq n_1$. We need, additionally, to know
that $\nabla_{X_i}Y_j=0$ for $i<n_1$ and each $j$, where for every $Y_j$ chosen in
the basis of $T_qM_2^{n_2}$, we take its horizontal lift on $T_{(p,q)}(M_1^{n_1}\times M_2^{n_2})$,
which we denote still by $Y_j$. Our setting corresponds now to [20, Proposition 56, p. 89].
Hence, $\nabla _{X}Y_j=0$.

Given the symmetries of $\nabla h$, it is enough to evaluate the following terms such as
$\nabla h(X_k,Y_i,Y_j)$, $\nabla h(Y_i,X_k,Y_j)$, $\nabla h(X,X_i,X_j)$, $\nabla h(Y,Y_i,Y_j)$
and $\nabla h(Y,X_i,X_j)$. Then finally by direct calculations we obtain $\nabla h=0$.
\end{proof}

\vskip 3mm

\begin{proof}[\textbf{Completion of the Proof of Theorem
\ref{thm:3.1}}]~

Let $\{X_i\}_{1\le i\le n_1}$ and $\{Y_j\}_{1\le j\le n_2}$ be the
local orthonormal frames of vector fields of $M_1^{n_1}$ and
$M_2^{n_2}$, respectively, as described in Lemma \ref{lem:3.5}.
Now, we consider the two distributions $\mathcal{D}_1$ spanned by
$X_1$, and $\mathcal{D}_2$ spanned by $\{X_2,\ldots,X_{n_1},
Y_1,\ldots,Y_{n_2}\}$. Given the form of $A_{JX_1}$ in
\eqref{eqn:3.28}, we may apply Theorem \ref{thm:2.1} and
obtain that $\psi: M^n\to \mathbb{CP}^n(4)$ is locally a Calabi
product Lagrangian immersion of an $(n-1)$-dimensional Lagrangian
immersion $\psi_1:M_{1,1}^{n-1}\to \mathbb{CP}^{n-1}(4)$ and a
point, i.e., $M^n=I_1\times M_{1,1}^{n-1}$, $I_1\subset \mathbb{R}$.
As $\psi$ is minimal in our case, we may further apply Theorem
\ref{thm:2.1} (2). Therefore, we get that
\begin{equation*}
\mu_1=\pm\frac{1}{\sqrt{n}}\, \text{ and }\psi_1 \text{ is minimal},
\end{equation*}
and $\psi=\Pi\circ\tilde{\psi}$ for
\begin{equation*}
\tilde{\psi}(t,p)=\Big(\sqrt{\tfrac{n}{n+1}}e^{i\frac{1}{n+1}t}
\tilde{\psi_1}(p), \sqrt{\tfrac{1}{n+1}}e^{-i\frac{n}{n+1}t}\Big), \
(t,p)\in I_1\times M_{1,1}^{n-1},
\end{equation*}
where $\Pi: \mathbb{S}^{2n+1}(1)\to\mathbb{CP}^{n}(4)$ is the Hopf
fibration and $\tilde{\psi}_1:M_{1,1}^{n-1}\to \mathbb{S}^{2n-1}(1)$
is the horizontal lift of $\psi_1$.

Consider next the immersion $\psi_1:M_{1,1}^{n-1}\to\mathbb{CP}^{n-1}(4)$.
From \eqref{eqn:3.28} we may see that the restriction  $A^1_J$ of
the shape operator $A_J$ on $\{X_2,\ldots,X_{n_1}, Y_1,\ldots,Y_{n_2}\}$
(which spans $TM_{1,1}^{n-1}$) is defined as
\begin{equation}\label{eqn:3.77}
\left\{
\begin{aligned}
A^1_{JX_2}X_2&=\lambda_{2,2}X_2,\\
A^1_{JX_i}X_i&=\mu_2 X_2+\cdots+\mu_{i-1} X_{i-1}+\lambda_{i,i}X_i,\ \ 3\le i\leq n_1,\\
A^1_{JX_i}X_j&=\mu_i X_j, \ \ 2\leq i\le j-1,\\
A^1_{JX_i}Y_j&=\mu_i Y_j,\ \ 2\leq i\leq n_1,\ \ 1\leq j\leq n_2,\\
A^1_{JY_i}Y_j&=\delta_{ij}(\mu_2 X_2+\cdots+\mu_{n_1} X_{n_1}).
\end{aligned}\right.
\end{equation}

We then apply Theorem \ref{thm:2.1} on $M_{1,1}^{n-1}$, by
identifying $\mathcal{D}_1$ with $\mathrm{span}\{X_2\}$ and $\mathcal{D}_2$
with $\mathrm{span}\{X_3,\ldots,X_{n_1}, Y_1, \ldots,Y_{n_2}\}$, and obtain
that $\psi_1:M_{1,1}^{n-1}\to \mathbb{CP}^{n-1}(4)$ is locally a
Calabi product Lagrangian immersion of an $(n-2)$-dimensional
Lagrangian immersion $\psi_2 :M_{1,2}^{n-2}\rightarrow
\mathbb{CP}^{n-2}(4)$ and a point, thus $M_{1,1}^{n-1}=I_2\times
M_{1,2}^{n-2}$ and $M^n=I_1\times I_2\times M_{1,2}^{n-2}$,
$I_2\subset \mathbb{R}$.

As $\psi_2$ is minimal, we further apply Theorem
\ref{thm:2.1} (2), and we get
\begin{equation*}
\mu_2=\pm \frac{1}{\sqrt{n-1}},\ \ \psi_2 \text{ is minimal},
\end{equation*}
and $\psi_1=\Pi_1\circ\tilde{\psi_1}$ for
\begin{equation*}
\tilde{\psi}_1(t,p)=\Big(\sqrt{\tfrac{n-1}{n}}
e^{i\tfrac{1}{n}t}\tilde{\psi}_2(p),
\sqrt{\tfrac{1}{n}}e^{-i\tfrac{n-1}{n}t}\Big),\ \ (t,p)\in I_2\times
M_{1,2}^{n-2},
\end{equation*}
where $\Pi_1: \mathbb{S}^{2n-1}(1)\to\mathbb{CP}^{n-1}(4)$ is the
Hopf fibration, and $\tilde{\psi}_2:M_{1,2}^{n-2}\to
\mathbb{S}^{2n-3}(1)$ is the horizontal lift of $\psi_2$.

In this way, we can apply Theorem \ref{thm:2.1} for the
$(n_1-1)^{th}$ time because, inductively, we have that
$\psi_{n_1-2}:M^{n-(n_1-2)}_{1,n_1-2} \to
\mathbb{CP}^{n-(n_1-2)}(4)$ is a Lagrangian immersion and the
restriction $A^{n_1-2}_J$ of the shape operator $A_J$ on
$\{X_{n_1-1},X_{n_1}, Y_1, \ldots,Y_{n_2}\}$ (which spans
$TM_{1,n_1-2}^{n-(n_1-2)}$) is defined as
\begin{equation}\label{eqn:3.78}
\left\{
\begin{aligned}
&A^{n_1-2}_{JX_{n_1-1}}X_{n_1-1}=\lambda_{n_1-1,n_1-1}X_{n_1-1},\\
&A^{n_1-2}_{JX_{n_1-1}}X_{n_1}=\mu_{n_1-1} X_{n_1},\\
&A^{n_1-2}_{JX_{n_1-1}}Y_j=\mu_{n_1-1} Y_j,\ \ 1\leq j\leq n_2,\\
&A^{n_1-2}_{JY_i}Y_j=\delta_{ij}\mu_{n_1-1} X_{n_1-1},\,\  1\leq i,
j\leq n_2.
\end{aligned}
\right.
\end{equation}

Then applying Theorem \ref{thm:2.2} by identifying $\mathcal{D}_1$
with $\mathrm{span}\{X_{n_1-1}\}$, and $\mathcal{D}_2$ with
$\mathrm{span}\{X_{n_1}, Y_1,\ldots,Y_{n_2}\}$, respectively, we
obtain that $M_{1,n_1-2}$ is locally a Calabi product Lagrangian
immersion of an $(n-(n_1-1))$-dimensional Lagrangian immersion $\psi_{n_1-1}
:M_{1,n_1-1}\rightarrow \mathbb{CP}^{n-(n_1-1)}(4)$ and a point.
Thus $M_{1,n_1-2}=I_{n_1-1} \times M_{1,n_1-1}$ and $M^n=I_1\times
I_2\times\cdots \times I_{n_1-1}\times M_{1,n_1-1}$,
$I_{n_1-1}\subset \mathbb{R}$.

As $\psi_{n_1-2}$ is minimal, we further apply Theorem
\ref{thm:2.1} (2) to see that
\begin{equation*}
\mu_{n_1-1}=\pm \frac{1}{\sqrt{n-(n_1-1)+1}},\ \ \psi_{n_1-1} \text{
is minimal},
\end{equation*}
and $\psi_{n_1-2}=\Pi_{n_1-2}\circ\tilde{\psi}_{n_1-2}$ for
\begin{equation*}
\begin{aligned}
\tilde{\psi}_{n_1-2}(t,p)=&\Big(\sqrt{\tfrac{n-(n_1-2)}{(n-(n_1-2))
+1}}e^{i\frac{1}{n-(n_1-2)+1}t}\tilde{\psi}_{n_1-1}(p), \\
&\ \
\sqrt{\tfrac{1}{n-(n_1-2)+1}}e^{-i\frac{n-(n_1-2)}{n-(n_1-2)+1}t}\Big),
\ (t,p)\in I_{n_1-1}\times M_{1,n_1-1}.
\end{aligned}
\end{equation*}
Here, $\Pi_{n_1-2}: \mathbb{S}^{2n-2n_1+5}(1)\to
\mathbb{CP}^{n-(n_1-2)}(4)$ is the Hopf fibration, and
$\tilde{\psi}_{n_1-1}:M_{1,n_1-1}\to \mathbb{S}^{2n-2n_1+3}(1)$ is
the horizontal lift of $\psi_{n_1-1}$.

We want to apply Theorem \ref{thm:2.1} for the $n_1^{th}$
time, for  the Lagrangian immersion
$\psi_{n_1-1}:M^{n-(n_1-1)}_{1,n_1-1}\to
\mathbb{CP}^{n-(n_1-1)}(4)$, given that the restriction
$A^{n_1-1}_J$ of the shape operator $A_J$ on $\{X_{n_1}, Y_1,
\ldots,Y_{n_2}\}$ (which spans $TM_{1,n_1-1}$) is defined as
\begin{equation}\label{eqn:3.79}
\left\{
\begin{aligned}
&A^{n_1-1}_{JX_{n_1}}X_{n_1}=\lambda_{n_1,n_1}X_{n_1},\\
&A^{n_1-1}_{JX_{n_1}}Y_j=\mu_{n_1} Y_j,\ \  1\leq j\leq n_2,\\
&A^{n_1-1}_{JY_i}Y_j=\delta_{ij}\mu_{n_1} X_{n_1}.
\end{aligned}
\right.
\end{equation}

Applying Theorem \ref{thm:2.2} again by identifying
$\mathcal{D}_1$ with $\mathrm{span}\{X_{n_1}\}$, and $\mathcal{D}_2$ with
$\mathrm{span}\{ Y_1,\ldots,Y_{n_2}\}$, we obtain that $\psi_{n_1-1}:
M_{1,n_1-1}\to \mathbb{CP}^{n-(n_1-1)}(4)$ is locally a Calabi
product Lagrangian immersion of an $(n-n_1)$-dimensional Lagrangian
immersion $\psi_{n_1} :M_{1,n_1}\rightarrow\mathbb{CP}^{n-n_1}(4)$
and a point. Thus $M_{1,n_1-1}=I_{n_1}\times M_{1,n_1}$ and
\begin{equation}\label{eqn:3.80}
M^n=I_1\times I_2\times\cdots\times I_{n_1}\times M_{1,n_1},\,
I_{n_1}\subset\mathbb{R}.
\end{equation}

As $\psi_{n_1-1}$ is minimal, we further apply Theorem
\ref{thm:2.1} (2) and we get
\begin{equation*}
\mu_{n_1}=\pm \frac{1}{\sqrt{n-n_1+1}},\ \ \psi_{n_1} \text{ is
minimal},
\end{equation*}
and $\psi_{n_1-1}=\Pi_{n_1-1}\circ\tilde{\psi}_{n_1-1}$ for
\begin{equation*}
\begin{aligned}
\tilde{\psi}_{n_1-1}(t,p)=&\Big(\sqrt{\tfrac{n-(n_1-1)}{(n-(n_1-1))
+1}}e^{i\frac{1}{n-(n_1-1)+1}t}\tilde{\psi}_{n_1}(p), \\
&\sqrt{\tfrac{1}{n-(n_1-1)+1}}e^{-i\frac{n-(n_1-1)}{n-(n_1-1)+1}t}\Big),
\ (t,p)\in I_{n_1}\times M_{1,n_1},
\end{aligned}
\end{equation*}
where $\Pi_{n_1-1}: \mathbb{S}^{2n-2n_1+3}(1)\to
\mathbb{CP}^{n-n_1+1}(4)$ is the Hopf fibration and
$\tilde{\psi}_{n_1}:M_{1,n_1}\to\mathbb{S}^{2n-2n_1+1}(1)$ is the
horizontal lift of  $\psi_{n_1}$.

Notice that  the restriction  $A^{n_1}_J$ of
the shape operator $A_J$ on $\{Y_1,\ldots,Y_{n_2}\}$ is $A^{n_1}_{JY_i}Y_j=0$.
Therefore, we eventually have that $M^n$ is locally
a Calabi product Lagrangian immersion of $n_1$ points
and an $n_2$-dimensional Lagrangian immersion
$$
\psi_{n_1}:M^{n_2}_2\to \mathbb{CP}^{n-n_1}(4),
$$
for $M^{n_2}_2:=M_{1,n_1}$ which has vanishing second fundamental
form. Moreover,
$$
M^n=I_1\times I_2\times \cdots\times I_{n_1}\times M^{n_2}_2,\ \
I_1,\ldots,I_{n_1}\subset \mathbb{R}.
$$

Finally, for $q\in M^{n_2}_2$ the parametrization of $\psi: M^n\to
\mathbb{CP}^n(4)$ is given by:
\begin{align*}
\psi(t_1,\ldots,t_{n_1},q)=\Big(&\sqrt{\tfrac{n-(n_1-1)}{n+1}}e^{i
\big(\frac{t_1}{n+1}+\frac{t_2}{n}+\cdots+\frac{t_{n_1-1}}{n-(n_1-2)+1}
+\frac{t_{n_1}}{n-(n_1-1)+1}\big)}\tilde{\psi}_{n_1}(q), \\
&\tfrac{1}{\sqrt{n+1}}e^{i\big(\frac{t_1}{n+1}+\frac{t_2}{n}
+\cdots+\frac{t_{n_1-1}}{n-(n_1-2)+1}-\frac{n-(n_1-1)}{n-(n_1-1)+1}t_{n_1}\big)},\\
&\tfrac{1}{\sqrt{n+1}}e^{i\big(\frac{t_1}{n+1}+\frac{t_2}{n}
+\cdots+\frac{t_{n_1-2}}{n-(n_1-3)+1}-\frac{n-(n_1-2)}{n-(n_1-2)+1}t_{n_1-1}\big)},\\
&\ \ \ \ldots\\
&\tfrac{1}{\sqrt{n+1}}e^{i\big(\frac{t_1}{n+1}+\frac{t_2}{n}-\frac{n-2}{(n-2)+1}t_3)},
\tfrac{1}{\sqrt{n+1}}e^{i(\frac{t_1}{n+1}-\frac{n-1}{n}t_2\big)},\\
&\tfrac{1}{\sqrt{n+1}}e^{-i\frac{n}{n+1}t_1}\Big),
\end{align*}
which, writing $\tilde{\psi}_{n_1}(q)=:(y_1,\ldots,y_{n_2+1})$, is
equivalent to
\begin{align}\label{eqn:3.81}
\begin{split}
\psi(t_1,\ldots,t_{n_1},q)=\Big(&\tfrac{1}{\sqrt{n+1}}e^{i
u_1},\ldots, \tfrac{1}{\sqrt{n+1}}e^{i u_{n_1}},\\
&\sqrt{\tfrac{n_2+1}{n+1}}e^{i u_{n_1+1}}\big(y_1, y_2,\ldots,
y_{n_2+1}\big)\Big),
\end{split}
\end{align}
where, $\{u_i\}_{1\le i\le n_1+1}$ are defined by
$$
\left\{
\begin{aligned}
&u_1=-\tfrac{n}{n+1}t_1,\\[-1mm]
&\ \ \ \ \ \ldots\\[-1mm]
&u_{n_1}=\tfrac{t_1}{n+1}+\tfrac{t_2}{n}+\cdots+\tfrac{t_{n_1-1}}{n-(n_1-2)
+1}-\tfrac{n-(n_1-1)}{n-(n_1-1)+1}t_{n_1},\\
&u_{n_1+1}=\tfrac{t_1}{n+1}+\tfrac{t_2}{n}
+\cdots+\tfrac{t_{n_1-1}}{n-(n_1-2)+1}+\tfrac{t_{n_1}}{n-(n_1-1)+1},
\end{aligned}
\right.
$$
and they satisfy $u_1+u_2+\cdots+u_{n_1}+(n_2+1)u_{n_1+1}=0$.

\vskip 1mm

This completes the proof of Theorem \ref{thm:3.1}.
\end{proof}

\subsection{}\label{sect:3.2}~

Now, we deal with Case (ii), that is, we treat the case when
$c_1\neq0$ and $c_2\neq 0$.

We begin with the following result whose proof is similar to that of
\eqref{eqn:3.4}.
\begin{lem}\label{lem:3.8}
If \textbf{Case (ii)} occurs, then we have
\begin{equation} \label{eqn:3.82}
\langle Y_l, A_{JX_i} X_j\rangle=\langle X_i, A_{JY_l} Y_m\rangle=0,
\ 1\leq i,j\leq n_1, \ 1\leq l,m\leq n_2.
\end{equation}
\end{lem}

Then, as main result of this subsection we can prove the following
lemma.

\begin{lem}\label{lem:3.9}
\textbf{Case (ii)} does not occur.
\end{lem}
\begin{proof}
Suppose on the contrary that Case (ii) does occur. From Lemma
\ref{lem:3.1} we know that $A_J$ vanishes nowhere. We may assume
that there exist $X\in T_pM_1^{n_1}$ such that $A_{JX}\neq0$ at the
point $p$. Given Lemma \ref{lem:3.8}, similarly to the proof of Lemma
\ref{lem:3.5}, we can show that there exists a local orthonormal
frame $\{X_1,\ldots,X_{n_1}\}\in TM_1^{n_1}$ on a neighbourhood of
$p$ such that the shape operator satisfies
\begin{equation}\label{eqn:3.83}
A_{JX_1}X_1=\lambda_1X_1,\ A_{JX_1}X_i=\lambda_2X_i, \ 2\le i\leq
n_1,
\end{equation}
where $\lambda_1$ and $\lambda_2$ are constants. Then, similarly to
the proof of \eqref{eqn:3.27}, we can show that ${\nabla}_XX_1=0$ for
any $X \in TM_1^{n_1}$. This implies that ${R}(X_1,X_2)X_1=0$, which
is a contradiction to $c_1c_2\neq0$.
\end{proof}

\vskip 3mm

\begin{proof}[\textbf{ Completion of the Proof of the Main Theorem}]~

If $c_1=c_2=0$, it follows from \eqref{eqn:2.12} that
$(M^n,\langle\cdot,\cdot\rangle)$ is flat. According to the result
of \cite{E,LZ} and \cite{CO} (see the Gauss equation (3.5) in
\cite{CO}), we get item (1) of the Main Theorem.

If $c_1^2+c_2^2\neq0$, we have two cases: Case (i) and Case (ii).

For Case (i), by Theorem \ref{thm:3.1}, we obtain the minimal
Lagrangian submanifold as stated in item (2) of the Main Theorem.

Whereas for Case (ii), by Lemma \ref{lem:3.9}, it does not occur.

Hence, we have completed the proof of the Main Theorem.
\end{proof}

\vskip 5mm


\vskip 1cm

\begin{flushleft}

Xiuxiu Cheng and Zejun Hu:\\
{\sc School of Mathematics and Statistics, Zhengzhou University,\\
Zhengzhou 450001, People's Republic of China.}\\
E-mails: chengxiuxiu1988@163.com; huzj@zzu.edu.cn.
\vskip 2mm

Marilena Moruz:\\
{\sc Department of Mathematics, KU Leuven, Celestijnenlaan 200B, Box 2400, BE-3001 Leuven, Belgium.}\\
E-mail: marilena.moruz@kuleuven.be.
\vskip 2mm

Luc Vrancken:\\
{\sc Universit\'e Polytechnique Hauts-de-France, F-59313 Valenciennes, France;
Department of Mathematics, KU Leuven, Celestijnenlaan 200B, Box 2400, BE-3001 Leuven, Belgium.}\\
E-mail: luc.vrancken@uphf.fr. 

\end{flushleft}

\end{document}